\documentclass[11pt]{amsart}
\usepackage{graphicx,color}
\usepackage[left=1in,right=1in,top=1in,bottom=1in]{geometry}
\usepackage{latexsym}
\usepackage{cite}
\usepackage{amsmath,amssymb,amsthm,mathrsfs}

\allowdisplaybreaks[1]  

\newtheorem{theorem}{Theorem}

\newtheorem{lemma}[theorem]{Lemma}
\newtheorem{defn}[theorem]{Definition}
\newtheorem{rem}[theorem]{Remark}

\renewcommand{\phi}{\varphi}

\let\epsilon=\varepsilon

\newcommand{\ppt}{\frac{\partial}{\partial t}}

\def\crn#1#2{{\vcenter{\vbox{
\hbox{\kern#2pt \vrule width.#2pt height#1pt
 }
\hrule height.#2pt}}}}

\newcounter{mnotecount}[section]
\let\oldmarginpar\marginpar
\renewcommand\marginpar[1]{\-\oldmarginpar[\raggedleft\footnotesize #1]%
{\raggedright\footnotesize #1}}

\begin{document}

%
%
%
%
%
%
%
%
%

\title{Scaling and Entropy for the RG-$2$ Flow}

\author[M. Carfora]{Mauro Carfora}

\address[Department of Physics, University of Pavia]{University of Pavia} 
\address[GNFM and INFN]{Italian National Group of Mathematical Physics, and INFN Pavia Section}
\email{mauro.carfora@unipv.it}

\author[C.Guenther]{Christine Guenther}

\address[Pacific University]{Pacific University} 

\email{guenther@pacificu.edu}

\thanks{This work was partially supported under a GNFM visiting professorship grant and Simons Foundation Collaboration Grant 283083}



\date{24 May 2018}

\begin{abstract}

Let $(M,g)$ be a closed Riemannian manifold. The \emph{second order approximation} to the perturbative renormalization group flow for the nonlinear sigma model (RG-2 flow) is given by :

\begin{equation*}
\label{standardRG2}
\frac{\partial }{\partial t} \, g(t) \, =\,  -2 \mathrm{Ric}(t) \, -\, \frac{\alpha}{2}%
\mathrm{Rm}^2(t),\newline
\end{equation*}
where
$ g  = \mathrm{Riemannian  \ metric}, 
\mathrm{Ric} = \mathrm{Ricci \ curvature, } \
\mathrm{Rm}^2_{ij}:=\mathrm{R}_{irmk}\mathrm{R}_j^{rmk},$
and $\alpha\ge 0$ is a parameter. The flow is invariant under diffeomorphisms, but not under scaling of the metric. 
We first develop a \emph{geometrically defined} coupling constant $\alpha_g$ that leads to an equivalent, scale-invariant flow. 
We further find a modified Perelman entropy for the flow, and prove local existence of the resulting variational system. The crucial idea is to modify the flow by two diffeomorphisms, the first being the usual DeTurck diffeomorphism the second being strictly related to the geometrical characterization of the coupling constant $\alpha_g$. We minimize the entropy functional so introduced to characterize a natural extension $\Lambda[g]$ of  the  Perelman's $\lambda(g)$--functional, and show that $\Lambda[g]$ is monotonic  under the RG-2 flow.
Although the modified Perelman entropy is monotonic, the RG-2 flow is not a gradient flow with respect this functional. We discuss this issue in detail, showing how to deform the functional in order to give rise to a gradient flow for a DeTurck modified RG-2 flow. 
\end{abstract}

\maketitle

\section{INTRODUCTION: A SCALE-INVARIANT RG-2 FLOW}
\label{Introduction}
The  RG-2 flow (see e.g. \cite{VB}, \cite{cremaschi}, \cite{GGI2}, \cite{GGI3}, \cite{GO},  \cite{oliy09})  is the geometric flow associated with the two--loop (\emph{i.e.} 
second order) approximation to the perturbative renormalization group flow for nonlinear sigma models \cite{mc},  \cite{DanAnnPhys}\, 
\cite{gawedzki}   given by  

\begin{equation} 
\begin{tabular}{l}
$\frac{\partial }{\partial t }g_{ij}(t )=\,-\,2\mathrm{Ric}_{ij}(t )\,-\,\frac{\alpha}{2}\,\mathrm{Rm}^2_{ij}(t),$ \\ 
\\ 
$g_{ab}(t =0)=g_{ab}$\;,%
\end{tabular}
   \label{RG2flow}
\end{equation}
where $\mathrm{Ric}(t )$, $\mathrm{Rm}(t )$ denote the Ricci and the Riemann  tensor of the evolving metric $g(t )$, and 
\begin{equation}
\mathrm{Rm}^2_{ij}(t)\,:=\,\mathrm{Rm}_{\;iklm}(t)\mathrm{Rm}_{\;jpqr}(t)g^{kp}(t)g^{lq}(t)g^{mr}(t)\;.
\end{equation} 
Note that the fixed parameter $\alpha\,\geq 0$ in (\ref{RG2flow}) is dimensionful (it has dimension of a length squared, \emph{i.e.} $[\alpha]\,=\,[L^2]$) and is typically assumed to be unrelated to the geometry. This immediately implies that the system of partial differential equations (\ref{RG2flow}) is not invariant under scalings of the metric: if $g \rightarrow \lambda g$,\, $\lambda\,\in\,\mathbb{R}_{>0}$\,,\, then $\mathrm{Rc}(\lambda g) = \mathrm{Rc}(g),$ but $\mathrm{Rm}^2(\lambda g) = \frac{1}{\lambda} \mathrm{Rm}^2(g)$, and consequently
\begin{equation}
2\mathrm{Ric}(\lambda g )\,+\,\frac{\alpha}{2}\,\mathrm{Rm}^2(\lambda g)\,=\,
2\mathrm{Ric}(g )\,+\,\frac{\alpha}{2}\,\lambda^{-1}\,\mathrm{Rm}^2(g)\;.
\end{equation}
 This is at variance with what happens for the Ricci flow, where one has manifestly parabolic space and time scaling symmetry, which are of basic importance in the geometric applications of the theory.  The lack of scaling invariance  is a source of a number of delicate problems in the analysis of  the RG-2 flow. This  is already evident when dealing with the condition assuring its (weak)--parabolicity, according to which the flow exists (and is parabolic) provided that \cite{GGI2}, \cite{cremaschi}, \cite{oliy09} 
  \begin{equation}
\label{parcond0}
1\,+\,\alpha\,\mathcal{K}_{(g)}[\sigma]\,>\,0\;,\;\;\;\;\;\forall \,\sigma\in Gr_{(2)}(TM)\;,
\end{equation}
where $\mathcal{K}_{(g)}[\sigma]$ denotes the sectional curvature of the initial $(M, g)$ along the plane $\sigma\in Gr_{(2)}(TM)$, and $Gr_{(2)}(TM)$ is the Grassmannian of 2-planes in $TM$;  however, under the scaling action $g \rightarrow \lambda g$ we get 
\begin{eqnarray}
\label{scalInv}
\mathcal{K}_{(\lambda g)}[\sigma(X,Y)]\,=\,\lambda^{-\,1}\,\mathcal{K}_{(g)}[\sigma(X,Y)]\;.
\end{eqnarray}  
It follows that if we assume that the condition (\ref{parcond0}) holds for the manifold $(M, g)$,  then on the rescaled manifold  $(M, \lambda g)$ the analogous condition may easily fail as soon as $\mathcal{K}_{(g)}[\sigma]\,< 0$ and $\lambda$ is small enough. \noindent The fact that (weak)--parabolicity of (\ref{RG2flow}) depends on the \emph{size} of the manifold is a somewhat unsatisfactory feature. One may argue that from a PDE point of view this behavior cannot be formally ruled out; nonetheless, one would like a deeper rationale for the fact that a geometric flow, driven by local curvatures, changes nature abruptly as a function of the  \emph{overall size} of the manifold. Moreover, the characterization of  the coupling $\alpha$ as a quantity unrelated to geometry is physically unjustified from the point of view of  the perturbative renormalization group for nonlinear sigma model, where one is forced to attribute the role of true coupling parameter to the normalized metric $\alpha^{-1}\,g$. This latter remark is made quite clear in D. Friedan's foundational paper  (see \cite{DanAnnPhys} pp.324 ) ,  where  (referring to the parameter $\alpha$ as a temperature $T$)  he stresses that  "\textit{... The temperature  $T$ in the coupling $T^{-1}g_{ij}$ is not a separate parameter. Multiplying $T$ by a positive constant $c$ while multiplying $g_{ij}$ by $c^{-1}$ leaves the coupling unchanged. The temperature is written separately only to make the expansion parameter visible and appears only in the combination $(Tg^{-1})^{ij}$. 
 ...}".  It must be said that, even if physically motivated, it is difficult to implement  Friedan's remark in the geometric flow framework associated with (\ref{RG2flow}), in particular if we want to preserve the locality requirements underlying quantum field theory.  
 
 Even if we put the locality requirements set by quantum theory aside, from the geometric point of view we need a natural mechanism that forces the rescaling of $\alpha$ along with  the rescaling of the metric, and this cannot be implemented with the RG-2 flow  as it stands. The most obvious candidate for such a mechanism, \emph{i.e.} setting $\left(\alpha\right)^{\frac{n}{2}}\,:=\,\int_M\,d\mu_g$, where  $d\mu_g$ is the Riemannian volume element, is not a viable prescription since along (\ref{RG2flow})  the Riemannian volume is not constant. To develop  a  solution to this problem, we exploit a natural variant of the Perelman's strategy \cite{Perelman}  by  introducting along the RG-2 flow a reference  measure, and associating to it  a \emph{geometrically defined} coupling constant $\alpha_{g}$. 
  
  More precisely, 
let $(M, g,\,d{\omega}(g))$ be a closed $n$--dimensional Riemannian manifold \,($n\geq 3$)\,with density \cite{grigoryan}, \cite{gromov}, \emph{i.e.} a smooth orientable manifold without boundary, endowed with a Riemannian metric $g$ and a Borel  measure $d{\omega}(g)$ that is absolutely continuous with respect to the Riemannian volume element $d\mu_g$. We set $d{\omega}(g)\,=\,e^{\,-{f}}\,d\mu_g$ for some smooth function ${f}\,\in\,C^\infty(M, \mathbb{R})$, and denote by  
\begin{equation}
\label{alphaG}
\left(\alpha_g\right)^{\frac{n}{2}}\,:=\,\int_M\,d{\omega}(g)\;
\end{equation}
the total $d{\omega}(g)$--mass of $(M, g)$, and by
\begin{equation}
\label{normomega}
d\widehat{\omega}(g)\,:=\,\left(\alpha_g\right)^{\,-\,\frac{n}{2}}\,d{\omega}(g)\,=\,\left(\alpha_g\right)^{\,-\,\frac{n}{2}}\,e^{\,-{f}}\,d\mu_g\;,\;\;\;\;\;\int_M\,d\widehat{\omega}\,=\,1\;,
\end{equation}
the associated probability measure $d\widehat{\omega}(g)$. Note that under the metric rescaling $g\,\longmapsto \,\lambda\,g$, \,$\lambda\,\in\,\mathbb{R}_{>0}$,\, the parameter $\alpha_g$, associated with $(M, g,\,d{\omega}(g))$ scales according to $\alpha_{\lambda g}\,=\,\lambda\,\alpha_g$.  Moreover, $\alpha_g$  is defined  up to the gauge transformation  
\begin{eqnarray}
\label{gauge1}
d\omega\,\longmapsto \,d\widetilde{\omega}\,&:=&\,d\omega\,+\,\alpha_g\,\mathrm{L}_{\xi_g}\,d\omega\\
&=&\left( 1\,+\, \alpha_g\,\mathrm{div}^{(\omega)}\,{\xi_g} \right)\,d\omega\nonumber\;,
\end{eqnarray}
where ${\xi_g}\,\in\,C^\infty(M, TM)$ is, as we shall see, a smooth gradient vector field, \textit{i.e.}, if  $g^{-1}$ denotes the inverse metric on $TM^*$,  $\xi_g\,=\,g^{-1}(d\psi,\cdot )$ for some smooth function $\psi$. (The notation ${\xi_g}$  emphasizes the metric dependence  induced by the gradient nature of ${\xi_g}$; also note that ${\xi_g}$ has the dimension of an inverse length, \emph{i.e.}\, $[\xi_g]\,=\,[L^{\,-1}]$). 
The weighted divergence operator $\mathrm{div}^{(\omega)}$ introduced in the gauge transformation (\ref{gauge1})  is characterized in terms of the Lie derivative $\mathrm{L}_{\xi_g}\,d\omega$ of $d\omega$ along  ${\xi_g}$ according to
\begin{equation}
\label{Wlie}
\mathrm{L}_{\xi_g}\,d\omega\,=\,\mathrm{div}^{(\omega)}{\xi_g}\,d\omega\;,
\end{equation}
or, in index notation,
\begin{equation}
\label{Wlieindex}
\mathrm{div}^{(\omega)}{\xi_g}\, =\,\nabla^{(\omega)}_k{\xi_g}^k\,:=\,\left(\nabla_k\,-\,\nabla_k f  \right){\xi_g}^k\;. 
\end{equation}
Hence, $\int_M\,d\widetilde{\omega}\,=\,\int_M\,\left( 1\,+\, \alpha_g \mathrm{div}^{(\omega)}\,{\xi_g} \right)\,d\omega\,=\,\int_M\,d\omega\,=\,\left(\alpha_g\right)^{\frac{n}{2}}$. 

In the analysis that follows, it is useful to keep track of the  gauge freedom  (\ref{gauge1})  by introducing the following characterization of the scale invariant  RG--$2$ flow. 
\begin{defn} (The scale--invariant RG--$2$ flow).\\
\label{SIRG2}
Let $[0,1] \ni t\longmapsto {\xi_{g(t)}}\in C^\infty(M, TM)$ be a given choice of a possibly $ t$--dependent vector field 
on $(M, g, d\omega)$. The scale--invariant RG--$2$ flow associated with the Riemannian manifold with density  $(M, g,\,d{\omega}(g))$  is 
 \begin{equation} 
	 \label{out1}
\frac{\partial }{\partial  t }g_{ij}( t)\,=\,-\,2\mathrm{Ric}_{ij}( t )\,-\,\frac{\alpha_g}{2}\,\mathrm{Rm}^2_{ij}( t)\;,\;\;\;\;\;g_{ij}( t =0)=g_{ij}\;, 
\end{equation}
coupled to the backward Fokker--Planck equation  describing the (backward) diffusion of the measure $d\omega(t)$ in presence of the drift generated by the given time--dependent  vector field ${\xi_{g(t)}}$,  
\begin{equation}
\label{out1b}
\frac{\partial}{\partial  t}\,d{\omega}( t)\,=\,-\,\Delta_{g( t)}\,d{\omega}( t)\,-\,
\mathrm{div}^{(\omega)}{\xi_{g(t)}}\,d\omega( t)\;,
\end{equation}
where  $\Delta_{g( t)}$ denotes the Laplace--Beltrami operator associated with the evolving metric $g( t )$ and $d\omega(t)$ is a shorthand notation for $d\omega[g(t)]$. 
\end{defn}

Note that when we couple (\ref{out1b})  to  (\ref{out1}), we can write $\alpha_{g(t)}$ in place of $\alpha_{g}$ since  under the flows   (\ref{out1}) and (\ref{out1b}) the coupling parameter $\alpha_{g(t)}$ remains constant.  Explicitly,  let us  introduce the $d{\omega}$--weighted Laplacian  \cite{grigoryan} on $(M, g,\,d{\omega})$ according to 
\begin{equation}
\label{weightLap}
\Delta_{g} ^{({\omega})}\phi\,:=\,
\Delta_{g}\,\phi\,-\,g^{ik}\nabla_i{f}\nabla_k \phi\;,\;\;\;\;\phi\in C^\infty(M, \mathbb{R})\;.
\end{equation} 
From the relation
\begin{equation}
\label{Wlap}
\Delta_g \,d\omega\,=\,\Delta_{g}\,\left(e^{\,-f} \,d\mu_{g}\right)\,=\,-\,\left(\Delta_g f - |\nabla f|_g^2 \right)\,e^{\,-f}d\mu_g\,
=:\,-\,\Delta _g^{(\omega)} f\,d\omega\;,
\end{equation}
we compute
\begin{eqnarray}
\frac{d}{d t}\,\left(\alpha_{g( t)}\right)^{\frac{n}{2}}\,&=&\,\frac{d}{d t}\,\int_M\,d\omega( t)\,=\,
\int_M\,\left[-\,\Delta_{g( t)}\,d{\omega}( t)\,-\,
\mathrm{div}^{(\omega)}{\xi}_{g(t)}\,d\omega( t)  \right]\nonumber\\
\label{alfac}\\
&=&\,-\,\int_M\,\Delta_{g( t)}\,(d{\omega}( t))\,=\,-\,\int_M\,\mathrm{L}_{\nabla\,f}\,d\omega( t)\,=\,0\nonumber\;.
\end{eqnarray}
In a sense, the evolution (\ref{out1b})  is a time--dependent version of the gauge transformation  (\ref{gauge1}).  Its  Fokker--Planck nature immediately follows since  the weighted divergence term  $\mathrm{div}^{(\omega)}{\xi}_{g(t)}\,d\omega( t)$  in  (\ref{out1b}) can be locally rewritten as $\nabla_k\left(\xi_g^k(t)\,d\omega(t)\right)$ which
 allows us to interpret  the  given time--dependent  vector field ${\xi_{g(t)}}$ as the generator of   a  drift acting on  
the (backward) diffusion of the measure $d\omega(t)$. The first set of results we present in this paper
are the local existence for the initial value problem associated with the coupled system (\ref{out1}) and (\ref{out1b}), and the characterization of its scaling properties. The second set of results concerns the proof of existence of a  monotonic functional $\mathcal{F}({g}(t), {{f}(t)}, {\xi}_g(t))$ (see \eqref {extPerelman}), which plays for the RG-2 flow the same role  that Perelman's $\mathcal{F}$--energy \cite{Perelman} has in standard Ricci flow theory. By minimizing  $\mathcal{F}({g}(t), {{f}(t)}, {\xi}_g(t))$ over the auxiliary fields $f$ and $\xi_g$ we are able to connect $\mathcal{F}({g}(t), {{f}(t)}, {\xi}_g(t))$  to a natural extension of  Perelman's $\lambda(g)$--functional,  and prove that this extended geometric functional,  $\Lambda[g]$, is monotonic along the RG--2 flow. The main theorems are Theorem  \ref{ThBB},  Theorem  \ref{ThBB2} and  Theorem \ref{Entropy}. These results were announced in \cite{G}.

\section{LOCAL EXISTENCE AND SCALE INVARIANCE}
 Local existence for the initial value problem associated with the coupled system
 (\ref{out1}) and (\ref{out1b}) directly follows from an obvious adaptation of the conditions  \cite{cremaschi}, \cite{GGI2} for weak--parabolicity for the standard RG-2 flow \eqref{RG2flow}. However, at variance with respect to its non--scaling behavior, we now have manifestly parabolic space and time scaling symmetry.  For the sake of notational clarity in addressing these scaling properties, in this section we shall explicitly write $d\omega[g(t)]$ in place of the shorthand notation $d\omega(t)$ we have been using.  
\begin{theorem}
\label{ThA}
Let $(M, g,\,d{\omega}[g])$ be a closed $n$--dimensional Riemannian manifold \,($n\geq 3$)\,with density, and denote by $Gr_{(2)}(TM)$ the Grassmannian of $2$--planes in $TM$. If the parameter 
$\alpha_g$ and the initial metric $g$ are such that
\begin{equation}
\label{parcond}
1\,+\,\alpha_g\,\mathcal{K}_{P}(M,g)\,>\,0\;,\;\;\;\;\;\forall \,P\in Gr_{(2)}(TM)\;,
\end{equation}
where $\mathcal{K}_{P}(M,g)$ is the sectional curvature of $(M,g)$ along the plane $P\in Gr_{(2)}(TM)$, then the initial value problem associated with (\ref{out1}) is weakly-parabolic, and there exists a unique solution
\begin{equation}
\left(t, g\right)\,\longmapsto\, g(t)\;,
\end{equation}
on some time interval $[0,\,T)$. Let $T_0\,<\,T$ and set $\eta\,:=\,T_0\,-\,t$. Then along the time--reversed flow $\eta\,\longmapsto \,g(\eta)$,\;$\eta\,\in\,[0, T_0]$,\, the evolution (\ref{out1b}) of the  measure $\eta\longmapsto d{\omega}[g(t=T_0-\eta)]\,=\,e^{\,-{f(\eta)}}\,d\mu_{g(\eta)}$, in the gauge defined by the chosen gradient vector field $\eta\longmapsto {\xi}_{g(\eta)}\,=\,\nabla_{g(\eta)}\,\psi(\eta)$, is governed by the Fokker--Planck equation 
\begin{equation}
\label{omegaheat}
\frac{\partial}{\partial\eta}\,d{\omega}[g(\eta)]\,=\,\Delta_{g(\eta)}\,d{\omega}[g(\eta)]\,+\,
 \mathrm{div}^{(\omega)}{\xi}_{g(\eta)}\,d\omega[g(\eta)]\;.
\end{equation}
The resulting  evolution $[0, T_0]\,\ni \,t\,\longmapsto (g(t),\,d{\omega}[g(T_0- t)])$ induces on the solution space of (\ref{out1}) and (\ref{out1b}) the parabolic space and time scaling symmetry
\begin{equation}
\label{scalingRG2}
\left(g(t),\, \xi_{g(t)},\, d{\omega}[g(T_0-t)] \right)\,\longmapsto \,
\left(\lambda\,g\left({t}/{\lambda} \right),\,\xi_{\lambda\,g(t/\lambda)},\,
d{\omega}\left[\lambda\,g\left({(T_0-t)}/{\lambda} \right)\right]   \right)\;,
\end{equation} 
for $t\,\in\,[0, T_0]$, and $\forall \,\lambda\,\in\,\mathbb{R}_{>0}$.
\end{theorem}

\begin{proof}
Since the coupling parameter $\alpha_g$ in  (\ref{out1}) refers to the given initial metric $g$, the proof of local existence of    (\ref{out1})  follows directly from the conditions \cite{cremaschi}, \cite{GGI2},  for weak--parabolicity for the standard RG-2 flow (\ref{RG2flow}) with $\alpha\, \equiv\, \alpha_g$.  Before proceeding with the analysis of  the  evolution (\ref{out1b}) of the  measure $t\longmapsto d\omega(t)$ along the solution
$(t, g)\,\longmapsto\,g(t)$ of  (\ref{out1}),  and of the associated   scaling properties (\ref{scalingRG2}),   we need to further explore the nature of the gauge choice associated with the drift $\xi_g$ in order to prove that it can always be assumed to be a gradient vector field, as anticipated. We start by noticing that (\ref{gauge1})  is an infinitesimal version of Moser's theorem \cite{moser}, and we can exploit the  Helmholtz decomposition  of the vector field $\xi_g$, to provide a finer resolution of the gauge freedom (\ref{gauge1}). Let us recall that the standard Helmholtz decomposition of a vector field on a Riemannian manifold $(M,g)$ is the orthogonal factorization, with respect to the standard $L^2(M,d\mu_g)$ inner product, of  a vector field into its gradient and  divergence--free part (see \cite{cantor}). To extend this factorization to the Riemannian manifold with density  $(M, g,\,d{\omega}(g))$, let $D_i\,:=\,e^{-f}\,\nabla_i$. Then for any smooth test function $\phi$ and vector field $X\,\in\,C^{\infty}(M,TM)$ we compute 
\begin{equation}
\label{weightCantor}
\int_M\,X^i\,D_i\phi\,d\mu_g\,=\,\int_M\,\left( \nabla_i(e^{-f}\,X^i) \right)\,\phi\,d\mu_g\,=\,\int_M\,e^{-f}\,\mathrm{div}^{(\omega)}\,X\,\phi\,d\mu_g\;,
\end{equation}
which shows that the formal  $L^2(M,d\mu_g)$ adjoint of $D$ is $D^*\,=\,e^{-f}\,\mathrm{div}^{(\omega)}$, or equivalently that the  formal adjoint of  $\nabla$, with respect to the $L^2(M, d\omega)$  inner product on $M$ associated to the measure $d\omega$,  is  the weighted divergence operator $\mathrm{div}^{(\omega)}$.  This latter remark and Theorem 3.12 of  \cite{cantor} directly imply that 
for any $1\,\leq\,q\,\leq\,s$, we have on the Riemannian manifold with density  $(M, g,\,d{\omega}(g))$ the weighted  Helmholtz decomposition
\begin{equation}
W_{(\omega)}^{p, q-1}(TM)\,=\,\mathrm{grad}\,\left(W_{(\omega)}^{p, q}(M) \right)\,\oplus \,\mathrm{Ker}\left(  \mathrm{div}^{(\omega)} \right)
\end{equation}
where, for $p\,>\,1$,\; $s\,>\,\frac{n}{p}\,+\,2$, \; $W_{(\omega)}^{p, s}(M)$ and  $W_{(\omega)}^{p, s}(TM)$   respectively denote the space of functions and vector fields of Sobolev class  $(p,s)$ with respect to the measure $d\omega$.   Hence, we can write
\begin{equation}
\xi_g\,=\,\nabla_g\,\psi\,+\,\xi^{\perp }_g,  
\end{equation}
where the vector field $\xi^{\perp }_{g}$ is such that  $\mathrm{div}^{(\omega)}(\xi^{\perp }_g)\,=\,0 $, and  $\nabla_g\,\psi\,:=\,g^{-1}(d\psi, \cdot )$ is the gradient of the scalar function  $\psi\,\in\, W_{(\omega)}^{p, q}(M)$ solution of the elliptic PDE
\begin{equation}
\label{ottoPDE}
\triangle ^{(\omega)}_g\,\psi\,=\, \mathrm{div}^{(\omega)}\,\xi_g\;,
\end{equation}
where  $\triangle ^{(\omega)}_g$  is the $d{\omega}$--weighted Laplacian  (\ref{weightLap})   on $(M, g,\,d{\omega})$.  It is perhaps  interesting to note that (\ref{ottoPDE}) can be interpreted as the Otto parametrization \cite{16}, \cite{17} of the tangent vectors 
\begin{equation}
\label{tangProb}
T_{\widehat{\omega}}\,\mathrm{Prob}_{ac}(M,g)\,:=\,\left\{h\,\in\,C^{\infty}(M)\,| \,\int_M\,h\,d\widehat{\omega}\,=\,0  \right\}\;,
\end{equation}
to the space of absolutely continuous probability measures $\mathrm{Prob}_{ac}(M,g)$  on $(M,g)$. It follows that the gauge transformation (\ref{gauge1})  can be equivalently rewritten as 
\begin{equation}
\label{gauge2}
d\omega\,\longmapsto \,d\widetilde{\omega}\,=\,\left( 1\,+\, \alpha_g\,\Delta ^{(\omega)}_g\,\psi \right)\,d\omega\;,
\end{equation}
in terms of the scalar function $\psi$. This gauge freedom is clearly defined up to the residual gauge characterized by the transformation $\nabla \psi\,\longmapsto \,\nabla \psi\,+\,\xi^\perp_g $, with $\mathrm{div}^{(\omega)}\,\xi^\perp_g \,=\,0$. Formally this can be justified on more sophisticated grounds by using the isomorphism between   $T_{\widehat{\omega}}\mathrm{Prob}_{ac}(M, g)$ and $C^{\infty}(M)/\mathbb{R}$.  This goes as follows.
Let  $h\,\in\,C^{\infty}(M)$\,with $\int_M\,h\,d\widehat{\omega}\,=\,0$ be the scalar function representing a tangent vector to $\mathrm{Prob}_{ac}(M, g)$ (see (\ref{tangProb})). Any such $h$ generates a gauge trasformation (\ref{gauge1})  according to
\begin{equation}
\label{gauge3}
d\widehat{\omega}\,\longmapsto \,d\widetilde{\omega}\,:=\,d\widehat{\omega}\,-\,h\,d\widehat{\omega}\;,
\end{equation}
where we have used the normalized probability measure  (\ref{normomega}) associated with $d\omega$, and the minus sign in front of $h$ is for later convenience. We can parametrize any $h\,\in\,T_{\widehat{\omega}}\mathrm{Prob}_{ac}(M, g)$ in terms of the solution $\varphi $ of the (Otto) elliptic partial differential equation
\begin{equation}
\Delta _g^{(\omega)}\,\varphi \,=\,-\,h\;,
\end{equation}
under the equivalence relation identifying any two such solutions differing by an additive constant. It is relatively easy to prove (\cite{16}  see also \cite{lott1b}) that the map so defined,
\begin{eqnarray}
T_{\omega}Prob_{ac}(M,g )\;\;\;\;\;\;\;&\longrightarrow & 
C^{\infty }(M ,\mathbb{R})/\mathbb{R}\; ,\\
h\;\;\;\;\;\;\;&\longmapsto &\;\;\;\;\; \varphi  \nonumber\;
\end{eqnarray}
is an isomorphism. Hence we can always write the general gauge trasformation (\ref{gauge3}) in the form
\begin{equation}
\label{gauge4}
d\widehat{\omega}\,\longmapsto \,d\widetilde{\omega}\,:=\,d\widehat{\omega}\,+\,\Delta _g^{(\omega)}\,\varphi \,d\widehat{\omega}\;,
\end{equation}
 which, up to the constant normalization factors related to $\alpha_g$, is exactly (\ref{gauge2}). 
\begin{rem}
According to these remarks, it follows that  we can always assume that the gauge vector field  $\xi_g$ is a gradient for some scalar function $\psi$, \textit{i.e.}, 
\begin{equation}
\xi_{g(t)}\,=\, \nabla_{g(t)}\,\psi(t)\;.
\end{equation}
However, for notational ease we do not emphasize the dependence from the potential $\psi$ since most of our evolutive equations are more easily expressed in terms of  $\xi_g$. The potential $\psi$ and the corresponding evaluation $\xi_{g}\,=\, \nabla_{g}\,\psi$ will be used only when the analysis requires it explicitly. Moreover,  by design, the prescription we adopt is to assign along $(t, g)\,\longmapsto\,g(t)$ the  vector field $\xi_g$  up to  an evolutive gauge condition of choice, (see Section \ref{sect5}),  and then solve the  evolution (\ref{out1b}) as a backward parabolic equation. Explicitly,  the local existence result for   (\ref{out1})  implies that for any $T_0\,<\,T$ we can consider the time--reversed flow 
$\eta\,\longmapsto \, g(\eta)$,\,$\eta\,:=\,T_0\,-\,t$,\, decorated by the given drift-generating time--dependent  vector field  $\xi_{g(t)}$  according to $\xi_{g(\eta)}\,:=\,\xi_{g(t\,=\,T_0-\eta)}$. The data $\left(g(\eta),\, \xi_{g(\eta)}\right)$ so defined
 characterize the  evolution (\ref{out1b}) of the  measure $d\omega$ as the solution  $\eta\longmapsto d{\omega}[g(\eta)]\,=\,e^{\,-{f(\eta)}}\,d\mu_{g(\eta)}$ of the forward parabolic Fokker--Planck equation
\begin{equation}
\label{omegaheat*}
\frac{\partial}{\partial\eta}\,d{\omega}[g(\eta)]\,=\,\Delta_{g(\eta)}\,d{\omega}[g(\eta)]\,+\,\mathrm{div}^{(\omega)}{\xi}_{g(\eta)}\,d\omega[g(\eta)]\;,
\end{equation}
 along the flow $\eta\,\longmapsto \, \left(g(\eta),\,\xi_{g(\eta)}\right)$,\,$\eta\,:=\,T_0\,-\,t$. 
\end{rem}
It is in such a framework that  the solution, 
$t\,\longmapsto \,\left( g(t),\,d\omega[g(T_0-t)]  \right)$, \,$t\,\in\,[0, T_0]$,\, of the coupled flows (\ref{out1}) and  (\ref{out1b}) exhibits a   \emph{scale invariance} as in the case of the Ricci flow.   Explicitly,
for the given gradient vector field $\xi_{g(t)}\,=\, \nabla_{g(t)}\psi(t)$,  let $\left(g,\,d\omega[g(\eta=0)],\,
\,\tilde{t}\,\right)\,\longmapsto\, \left(g(\,\tilde{t}\,) ,\,d\omega[g(T_0-\,\tilde{t}\,) ] \right)$, \;$\,\tilde{t}\,\in [0, \widetilde{T}_0)$,\, be a solution of the coupled RG-$2$ flows  (\ref{out1})  and  (\ref{out1b}) in the forward $\,\tilde{t}\,$ and backward $\widetilde{\eta}\,:=\,\widetilde{T}_0\,-\,\tilde{t}$  evolution times, \emph{i.e.},   \\ 
\begin{equation} 
\frac{\partial }{\partial \,\tilde{t}\, }\;g_{ij}(\,\tilde{t}\, )\,=\,-\,2\mathrm{Ric}_{ij}\left(g(\,\tilde{t}\,) \right)\,-\,\frac{\alpha_{g}}{2}\,\mathrm{Rm}^2_{ij}(\,\tilde{t}\,)\;,\;\;\;\; g_{ab}(\,\tilde{t}\, =0)=g_{ab}\;,
   \label{RGflowAA}
\end{equation}
and 
\begin{equation}
\label{omegaheatilde}
\frac{\partial}{\partial\widetilde{\eta}}\,d{\omega}[g(\widetilde{\eta})]\,=\,\Delta_{g(\widetilde{\eta})}\,d{\omega}[g(\widetilde{\eta})]\,+\,
\mathrm{div}^{(\omega)}\xi_{g(\widetilde{\eta})}\,d\omega[g(\widetilde{\eta})]\;.
\end{equation}
Let us rescale $\tilde{t}$ according to $\,\tilde{t}\,\,=\,\frac{t}{\lambda}$, $\lambda\in\mathbb{R}_{>0}$. We get \\
\begin{equation}
 \label{RGflowbisAA}
\frac{\partial }{\partial t}\,\lambda\,g_{ij}\left({t}/{\lambda}\right)\,=\,-\,2\mathrm{Ric}_{ij}\left(g\left({t}/{\lambda}\right) \right)\,-\,\frac{\alpha_{g}}{2}\,\mathrm{Rm}^2_{ij}\left(g\left({t}/{\lambda}\right)\right)\;,\;\;\;\;g_{ab}\left({t}/{\lambda}=0\right)=g_{ab}\;,
\end{equation}
and
\begin{equation}
\label{omegarescaled}
\frac{\partial}{\partial \eta}\,\lambda\,d{\omega}[g\left( {\eta}/{\lambda} \right)]\,=\,\Delta_{g ({\eta}/{\lambda}) }
\,d{\omega}[g\left({\eta}/{\lambda}\right)] \,+\,
\mathrm{div}^{(\omega)} \xi_{g({\eta}/{\lambda})}\,d\omega[g \left({\eta}/{\lambda}\right)]\;,
\end{equation}
where we have rewritten
${\partial }/{\partial ({t}/{\lambda}) }$ as $\lambda\,{\partial }/{\partial t}$ (similarly for ${\partial }/{\partial ({\eta}/{\lambda}) }$).
We start the discussion of the scaling properties of  (\ref{omegarescaled}) by rewriting it as 
\begin{equation}
\label{omegarescaledbis}
\frac{\partial}{\partial \eta}\,d{\omega}[g\left( {\eta}/{\lambda} \right)]\,=\,\lambda^{-1}\,\Delta_{g ({\eta}/{\lambda}) }
\,d{\omega}[g\left({\eta}/{\lambda}\right)] \,+\,
\mathrm{div}^{(\omega)} \lambda^{-1}\xi_{g({\eta}/{\lambda})}\,d\omega[g \left({\eta}/{\lambda}\right)]\;,
\end{equation}
and by noticing that under the rescaling $g\left({\eta}/{\lambda} \right)\longmapsto \lambda\,g\left({\eta}/{\lambda} \right)$ the Laplace--Beltrami operator $\Delta_{g (\frac{\eta}{\lambda}) }$,  the gradient vector field $\xi_{g(t)}\,=\,\nabla_{g(t)}\psi(t)$,
and the measure $d{\omega}[g\left({\eta}/{\lambda} \right)]$ scale as
\begin{eqnarray}
 \Delta_{\lambda\,g ({\eta}/{\lambda}) } \,&=&\,\lambda^{-1}\, \Delta_{g ({\eta}/{\lambda}) }\;,\nonumber\\
\label{sclDM}\\
\xi_{\lambda\,g({\eta}/{\lambda})}^i\,&=&\,\lambda^{-1}\, g^{ik}({\eta}/{\lambda})\,\partial_k\,\psi({\eta}/{\lambda})\,=\,
\lambda^{-1}\,\xi_{g({\eta}/{\lambda})}^i\;,\nonumber\\
\nonumber\\
d{\omega}[\lambda\,g\left({\eta}/{\lambda} \right)]\,&=&\,\lambda^{\frac{n}{2}}\,d{\omega}[g\left({\eta}/{\lambda} \right)]\;.\nonumber
\end{eqnarray}
From these relations it follows that if  along $t\,\longmapsto \,g(t)$  we rescale the metric according to  
$g_{ab}(\eta)\,\longmapsto \, \lambda\,g_{ab}\left(\frac{\eta}{\lambda} \right)$  then (\ref{omegarescaledbis}) reduces to
\begin{equation}
\label{omegarescaled*}
\frac{\partial}{\partial \eta}\,d{\omega}[\lambda\,g\left({\eta}/{\lambda} \right)]\,=\,\Delta_{\lambda\,g ({\eta}/{\lambda}) }
\,d{\omega}[\lambda\,g\left({\eta}/{\lambda}\right) ]\,+\,
\mathrm{div}^{(\omega)} \xi_{\lambda\,g({\eta}/{\lambda})}\,d\omega[g \left({\eta}/{\lambda}\right)]\;.
\end{equation}
According to  (\ref{alphaG}) and  (\ref{alfac}), the  scaling relation (\ref{sclDM}) for the measure $d{\omega}$ also implies that
\begin{equation}
\alpha_{\lambda\,g}\,=\,
\alpha_{\lambda\,g({\eta}/{\lambda})}\,=\,\lambda\,\alpha_{g({\eta}/{\lambda})}\,=\,\lambda\,\alpha_{g}\;.  
\end{equation}
If we take into account this latter result and the Riemannian scaling relations
$\mathrm{Ric}\left(g\left(\frac{t}{\lambda}\right) \right)\,=\,\mathrm{Ric}\left(\lambda\,g\left(\frac{t}{\lambda}\right) \right)$, \; 
$\mathrm{Rm}^2\left(\lambda g(\frac{t}{\lambda})\right)=\lambda^{\,-1}\,\mathrm{Rm}^2\left( g(\frac{t}{\lambda})\right)$,\; 
it easily follows that the rescaled metric  $(\lambda\,g_{(0)},\,t)\longmapsto \lambda\,g\left(\frac{t}{\lambda} \right)$,\; ${t}\in [0, \lambda\,\widetilde{T})$ is a space and time rescaled solution of the RG-$2$ flow\\

\begin{equation} 
\begin{tabular}{l}
$\frac{\partial }{\partial t}\, \lambda\,g_{ij}\left(\frac{t}{\lambda}\right)=\,-\,2\mathrm{Ric}_{ij}\left(\lambda\,g\left(\frac{t}{\lambda}\right) \right)\,-\,\frac{\alpha_{\lambda g}}{2}\,\mathrm{Rm}^2_{ij}\left(\lambda g(\frac{t}{\lambda})\right)\,,$ \\ 
\\ 
$\lambda\,g_{ab}\left(\frac{t}{\lambda}=0\right)\,=\,\lambda\,g_{ab}\;,\;\;\;\;\;\;{t}\in [0, \lambda\,\widetilde{T})$\;.%
\end{tabular}
   \label{RGflowtrisAA}
\end{equation}
This, together with (\ref{omegarescaled*}), implies that   the solution of the coupled system (\ref{out1}) and (\ref{out1b}) 
has the parabolic space and time scaling symmetry
\begin{equation}
\left(g(t),\, \xi_{g(t)},\, d{\omega}\left[g\left(T_0-t\right)\right]\right)\,\longmapsto \,
\left(\lambda\,g\left({t}/{\lambda} \right),\,\xi_{\lambda\,g(t/\lambda)},\,
d{\omega}\left[\lambda\,g\left({(T_0-t)}/{\lambda} \right)\right]   \right)\;,
\end{equation} 
$t\,\in\,[0, T_0]$,\;$\forall \,\lambda\,\in\,\mathbb{R}_{>0}$,\;as stated.
\end{proof}

We briefly elaborate on the consequences of the scale-invariant flow to solitons.
In \cite{GGI}, the authors investigated soliton structures for the RG-2 flow, and (without the scale-invariant $\alpha$) concluded that homothetically expanding solitons were quite restricted; for example, they are Einstein manifolds. One may suspect that these restriction are somehow related to the lack of scale invariance of  (\ref{RG2flow}), but this is not the case since this behavior holds also for the scale invariant RG-2 flow.   We show this explicitly in the case that $g_0$ has constant curvature. Let $\left(M, g_0, d\omega(g_0)\,=\,e^{-\,f_0}\,d\mu_{g_0}\right)$ be the Riemannian manifold 
 that we use as initial datum (in the sense specified by Theorem \ref{ThA}) for the coupled system (\ref{out1}) and (\ref{out1b}) defining the scale-invariant RG-2 flow, and let us assume that $g_0$ is a constant curvature metric, \, \emph{i.e.} $R_{ijkl}(g_0)\,=\,K\,\left((g_0)_{il} (g_0)_{jk}-(g_0)_{ik}(g_0)_{jl}\right)$ for some constant $K$. Say one has a solution $t\,\mapsto\,\left(g(t), d\omega(t)\,=\,e^{-\,f(t)}\,d\mu_{g(t)}\right)$ of  (\ref{out1}) and (\ref{out1b}) where the metric part evolves by scaling,  \emph{i.e.}, $g(t) = \sigma(t) g_0$, with  $\sigma(t)\,\in\,\mathbb{R}_{>0}$, \,$\sigma(t=0)\,=\,1$.  If we take into account these remarks, then
\begin{equation}
-2\mathrm{Ric}(g(t)) - \frac{\alpha_{g(t)}}{2} \mathrm{Rm}^2(g(t)) = -2\mathrm{Ric}(g_0) - \frac{\alpha_{g_0}}{2}\,\sigma(t)^{-\,1}\,
 \mathrm{Rm}^2(g_0)\;,
\end{equation}
where, besides $\alpha_{g_0}\,=\, \alpha_{g(t)}$, we have used $\mathrm{Ric}(g(t)) = \mathrm{Ric}(g_0)$ and $\mathrm{Rm}^2(g(t))\,=\,\sigma(t)^{-\,1}\,\mathrm{Rm}^2(g_0)$. Since $g_0$ is of constant curvature, we can write $\mathrm{R}_{ij}(g_0)\,=\,K\,(n-1)\,(g_0)_{ij}$ and $\mathrm{Rm}^2_{ij}(g_0)\,=\,2K^2\,(n-1)\,(g_0)_{ij}$.
Hence, corresponding to the assumed scaling evolution for the metric, the RG-2 flow takes the form (factoring out a common $g_0$)
\begin{equation}
\frac{d}{d t}\sigma(t) \,= \, -2K\,(n-1)\, -\alpha_{g_0}\,\sigma(t)^{-\,1}\, K^2\,(n-1)\,, 
\end{equation}
which is  implicitly solved  by  the Lambert W function construction that was developed in \cite{GGI3}, \emph{i.e.},  by scaling factors $\sigma(t)$ which satisfy 
\begin{equation}
\sigma(t) \,= \, -2K\,(n-1)t\,+\,1\, +\,\frac{\alpha_{g_0}\,K}{2}\,\ln\,\left|\frac{2\sigma(t)\,+\,\alpha_{g_0}K}{2\,+\,\alpha_{g_0}K} \right|\;.
\end{equation}

These remarks explicitly show that the variegated nature of the soliton structures for the standard RG-2 flow is not caused by the fact that the flow is not scale invariant. The complex structure is indeed found also for the scale-invariant RG-2 flow. It is the geometric interplay with the $\alpha_g\mathrm{Rm}^2$ term that claims responsibility for that.

\section{ENTROPIES}

A second set of results we prove concerns the existence of a monotonic functional which plays for  (\ref{out1}) the same role Perelman's $\mathcal{F}$--energy \cite{Perelman} has in standard Ricci flow. This is a very delicate issue which has two distinct aspects. One concerns to what extent the Perelman's functional $\mathcal{F}$ may be used to control also the RG-2 flow, an issue that in the physics literature has been addressed at various levels by A. Tseytlin \cite{tseytlin} and by T. Oliynyk, V. Suneeta, and  E. Woolgar \cite{oliy}, in connection with A. Zamolodchikov's \textit{c-theorem} \cite{zamo}. The other issue concerns the possibility of extending Perelman's technique for  constructing  explicitly a monotonic functional with respect to which the RG-2 flow is gradient. An entropy for a (normalized) RG-2 flow on surfaces with positive curvature was found by V. Branding \cite{VB}, by generalizing R. Hamilton's entropy for the Ricci flow on surfaces with positive curvature \cite{Ha}; however, as in the Ricci flow case, this RG-2 flow surface entropy does not generalize to higher dimensional manifolds.  
It is interesting to note that in \cite{flowers}, B. Chow and P. Lu considered an approach to entropy for the RG-2 flow in general dimensions, with the hope that it would apply, in some recursive way, also to higher loops corrections (see \cite{flowers}, equation (17.32)). The functional that they consider is the natural analog of Perelman's functional. They were able to derive a quantity based on this functional which, at a \emph{given fixed time}, is instantaneously monotonic if one considers the sum of the instantaneous and synchronous variation of the Perelman functional along the Ricci flow direction and along a $\mathrm{Rm}^2$ flow direction. 

 In what follows, we derive an extended Perelman entropy $\mathcal{F}({g}(t), {{f}(t)}, {\xi}_g(t))$
that is a natural generalization of Perelman's entropy by exploiting the gauge freedom related to the gradient vector field ${\xi_g}$. Although our strategy emphasizes, as in Perelman's analysis of the Ricci flow \cite{Perelman}, the interplay between the  diffeomorphism group and the RG-2 flow, it has aspects that are in the spirit of Chow and Lu's suggestion. We replace their two--flows splitting with the full RG-2 flow coupled to a corresponding auxiliary flow governing the gauge drift vector field ${\xi_g}$ associated to the measure $d\omega$. It is the latter that allows to take into account the contribution of the $\mathrm{Rm}^2$ term  to the entropy. Quite remarkably, by minimizing  $\mathcal{F}({g}(t), {{f}(t)}, {\xi}_g(t))$ with respect to the auxiliary fields $f$ and $\xi_g$, we obtain a geometric functional $\Lambda[g]$ directly related to a natural extension of Perelman's $\lambda(g)$--functional in terms of the first eigenvalue of the weighted Laplacian $\Delta_g^{({\omega})}$ on $(M, g,  d \omega)$.  As a consequence of the monotonicity of $\mathcal{F}({g}(t), {{f}(t)}, {\xi}_g(t))$  we prove that this extended  $\Lambda[g]$ is monotonic for the RG-2 flow.\\
\\
 To begin, let us recall that in the Ricci flow case, Perelman's energy functional  is constructed  by considering, along the Ricci flow metric $[0,\,T_0]\,\ni\,t\,\longmapsto \,h(t)$, solution of 
 \begin{equation}
\frac{\partial}{\partial t}\,h(t)\,=\, -2\,\mathrm{Ric}(h(t))\;,\;\;\;\;\;\;\;\; h(0)\,=\,h_0\;,
\label{RiccioA}
\end{equation}
 a (probability) measure $d\pi(t)\,:=\,e^{\,-\,{m(t)}}\,d\mu_{t}$, \, $m(t)\,\in\,C^\infty(M, \mathbb{R})$,  evolving according to  the backward heat equation
\begin{equation}
\label{BackHeat}
\frac{\partial}{\partial  t}\,d\pi( t)\,=\,
-\,\Delta_{h( t)}\,d\pi ( t)\;.
\end{equation}
To the resulting $t\,\longmapsto \,\left(M, h,\, d{\pi }\right)$ we associate the $\mathcal{F}(h(t), {m(t)})$--energy functional 
\begin{equation}
\label{perelman1}
\mathcal{F}(h(t), {m(t)})\,:=\,\int_M\,\left[\mathrm{R}(h(t))\,+\,|\nabla m(t)|^2_{h(t)} \right]\,d\pi (t)\,=\,
\int_M\, \mathrm{R}^{Per}(h(t))\,d\pi (t)\;,
\end{equation}
where 
\begin{equation}
\label{defPerlmanR}
\mathrm{R}^{Per}\,:=\,\mathrm{R}\,+\,2 \Delta _{h}f\,-\,\left|\nabla f \right|_h^2\,=\mathrm{R}\,+\,2 \Delta^{(\omega)} _{h}f\,+\,\left|\nabla f \right|_h^2
\end{equation}
denotes  the Perelman's modified scalar curvature associated with the Riemannian manifold with density $(M, h, d\pi)$.
$\mathcal{F}(h(t), {m(t)})$ is the \emph{entropy production functional}  $\frac{d}{d t}N(h(t), d\pi (t))$ of the related entropy (Nash entropy)   
\begin{equation}
\label{NashE1}
N(h(t), d\pi (t))\,:=\,-\int_M\,\log\left(\frac{d\pi (t)}{d\mu_{h(t)}}\right)\,d\pi (t)\;,
\end{equation} 
associated with the coupled evolution  (\ref{RiccioA}) and (\ref{BackHeat}). As a consequence of  the time--dependence of the metric $h(t)$,  the Nash entropy is not a monotonic quantity, whereas the entropy production functional $\mathcal{F}(h(t), {m(t)})$ turns out to enjoy a subtle monotonicity property of great geometrical relevance. This was one of Perelman's fundamental discoveries \cite{Perelman}. If, along the flow $t\,\longmapsto \,(h(t), d\pi(t))$, \,$t\,\in\,[0, T_0]$,\, one considers the $1$--parameter family of diffeomorphisms $\phi(t)\,:\,M\,\longmapsto \,M$ solution of the system of ODE $\frac{d}{dt}\,\phi(t)\,=\,-\,\nabla _{h(t)}\,m(t)$, \,$\phi(t=0)\,=\,\mathrm{id}_M$, then the pulled back metric and measure, 
 $\overline{h}(t)\,:=\,\phi(t)^*h(t)$ and $d\overline{\pi}(t)\,:=\,\phi(t)^*d\pi(t)$ satisfy the system
 \begin{eqnarray}
\frac{\partial}{\partial t}\,\overline{h}(t)\,&=&\, -2\,\left(\mathrm{Ric}(\overline{h}(t))\,+\, \nabla _{h(t)}\nabla _{h(t)}\,m(t)\circ \phi(t) \right)\;,\;\;\;\;\;\;\;\; \overline{h}(0)\,=\,h_0\;,\nonumber\\
\label{RiccioA2}\\
\frac{\partial}{\partial t}\,d\overline{\pi}(t)\,&=&\,0\;,\nonumber
\end{eqnarray}
which appears as the gradient flow of the functional (\ref{perelman1}). Note that by  diffeomorphim invariance, one easily shows that 
$\mathcal{F}(h(t), {m(t)})$ is monotonic along the original coupled flows (\ref{RiccioA}) and (\ref{BackHeat}), 
\begin{equation}
\frac{d}{dt}\,\mathcal{F}(h(t), {m(t)})\,=\,2\,\int_M\,\left|\mathrm{R}(h(t))\,+\,|\nabla m(t)|^2_{h(t)} \right|^2\,d\pi (t)\;.
\end{equation}
The minimization of  $\mathcal{F}(h(t), {m(t)})$ over all possible  (absolutely continuous) probability measures $d\pi $ provides Perelman's  geometric functional $\lambda(h)$  (see equation (\ref{Perelambda}) below for the explicit definition) which is monotonically non-decreasing along the  Ricci flow.  

\section{MONOTONICITY OF THE NASH ENTROPY}

 Not surprisingly, the situation described above is significantly more complex for the   RG--2 flow (\ref{out1}).  To begin with, if  we choose the gauge vector field ${\xi_{g(t)}}\,=\,\nabla_{g(t)}\psi(t)\,\equiv \,0$  for all $t\in [0, T_0]$,  then from the parabolicity requirement for the RG-2 flow, we have monotonicity  for a modified Nash entropy. 
\begin{theorem}
\label{ThB}
Let $T_0\,<\,T$ and, along the flow $[0, T_0]\,\ni t\,\longmapsto \left(g(t), d{\omega}(t);\,{\xi_{g(t)}}\,=\,0\right)$ solution of the RG--2 flow (\ref{out1}), define the extended Nash entropy functional 
\begin{eqnarray}
\label{NashE}
{N}\left(g(t), d{\omega}(t)\right)\,&:=&\,-\int_M\,\log\left(\frac{d\omega (t)}{d\mu_{g(t)}}\right)\,d\omega (t)\,
-\,n(n-1)\,\alpha_g^{\frac{n}{2}\,-\,1}\,t\nonumber\\
\\
&=&\,-\,\int_M\,\left(f(t)\,+\,\frac{n(n-1)\,t}{\alpha_g}\,\right)\,\,e^{\,-\,f(t)}d\mu_{g(t)}\;.\nonumber
\end{eqnarray} 
Then, as long as  $1\,+\,\alpha_g\,\mathcal{K}_{P}(t)\,>\,0$,\,$\forall \,P\in Gr_{(2)}(TM)$, we have 
\begin{eqnarray}
&&\frac{d}{d t}\,{N}\left(g(t), d{\omega}(t)\right)\nonumber\\
\label{entrop}\\
&&=\,\int_M\,\left[\mathrm{R}^{Per}(g(t))\,+\,\frac{\alpha_g}{4}\left|\mathrm{R}m(g(t))\right|^2_{g(t)}\,+\, 
\frac{n(n-1)}{\alpha_g}\right]\,e^{\,-\,f(t)}\,d\mu_{g(t)}\,\geq\,0\;.\nonumber
\end{eqnarray}
\end{theorem}
\begin{proof}
The gauge choice  $\xi_{g(t)}\,=\,0$ for all $t\in [0, T_0]$   uncouples the evolution of the measure $d\omega(t)$ from $\xi_{g(t)}$, and if we
compute, along (\ref{out1}) and (\ref{out1b}),  the derivative $\frac{d}{d t}\,{N}\left(g(t), d{\omega}(t)\right)$ of  the relative entropy functional defined by  (\ref{NashE}) , we get 
\begin{eqnarray}
\frac{d}{dt}\,N\left(g(t), d{\omega}(t)\right)\,&=&\,-\,\frac{d}{d t}\,\int_M\,f\,e^{-f}\,d\mu_g\,+\,
\frac{n(n-1)}{\alpha_g}\,\int_M\,e^{-f}\,d\mu_g\nonumber\\
\label{comp1}\\
&=&\,-\int_M\,\frac{\partial f}{\partial t}\,e^{-f}\,d\mu_g\,-\,\int_M\,f\,\frac{\partial }{\partial t}\,\left(e^{-f}\,d\mu_g\right)\,
+\,\frac{n(n-1)}{\alpha_g}\,\int_M\,e^{-f}\,d\mu_g\nonumber\;,
\end{eqnarray}
where we dropped all 
$t$--dependence, since notation wants to travel light.
\noindent From  (\ref{Wlap}), one recovers the standard relation
\begin{equation}
\frac{\partial f}{\partial t}\,=\,-\,\Delta_g f + |\nabla f|_g^2 \,+\,\frac{1}{2}g^{ab}\frac{\partial g_{ab}}{\partial t}\,=\,-\Delta_g^{({\omega})}\,f\,+\,\frac{1}{2}g^{ab}\frac{\partial g_{ab}}{\partial t}\;,
\end{equation} 
where  $\Delta_g^{({\omega})}$ is the $d{\omega}(t)$--weighted Laplacian on $(M, g(t),  d \omega(t))$.  We also need the identity (integration by parts)
\begin{eqnarray}
\int_M\,f\, \Delta_g \left(e^{\,-f} \right)\,d\mu_g\,&=&\,\int_M\,\Delta_g\,f  \,e^{\,-f}d\mu_g\,=\,
\int_M\,\left(\Delta_g^{({\omega})}\,f \,+\,  |\nabla f|_g^2 \right)\,e^{\,-f}d\mu_g\\
&=&\,
\int_M\,|\nabla f|_g^2 \,e^{\,-f}d\mu_g\;.\nonumber
\end{eqnarray}
Introducing these expressions in (\ref{comp1})  we  get
\begin{equation}
\label{}
\frac{d}{d t}\,N\left(g(t), d{\omega}(t), t\right)\,
=\,\int_M\,\left[-\,\frac{1}{2}g^{ab}\frac{\partial g_{ab}}{\partial t}\,+\, |\nabla f|_g^2   \,+\, \frac{n(n-1)}{\alpha_g}\right]\,e^{\,-f}d\mu_{g(t)}\;,
\end{equation}
where everything depends on $t$, and which along the RG--2 flow  (\ref{out1}) provides
\begin{equation}
\label{entrop2}
\frac{d}{d t}\,N\left(g(t), d{\omega}(t), t\right)\,
=\,\int_M\,\left[\mathrm{R}\,+\,\alpha_g|\mathrm{Rm|}^2_{g}\,+\,|\nabla f|_g^2\,+\, \frac{n(n-1)}{\alpha_g}\right]\,e^{\,-f}d\mu_{g(t)}\;.
\end{equation}
\\
\noindent
 Let us now observe that at any given point $x\in M$ we can rewrite the scalar curvature $\mathrm{R}(x, t)$ in terms of the sectional curvatures $\mathcal{K}_{P}(x, t)$ of $(M, g(t))$ as the $2$--planes $P$ vary in the Grassmannian $Gr(2)(T_x M)$. To this end, if we let $\left\{e_{(a)} \right\}_{a=1}^n$ denote an orthonormal basis for $T_x M$, and denote by $P(a,b)$ the $2$--plane in $Gr(2)(T_x M)$ generated by $e_{(a)}\wedge e_{(b)}$, \, with $a\not= b$, then
\begin{equation}
\mathrm{R}(x, t)\,=\,\sum_{a, b =1, a\not= b }^n\,\mathcal{K}_{P(a,b)}(x, t)
\end{equation}
Since $\sum_{a, b =1, a\not= b}^n\,1\,=\,n(n-1)$ we can write
\begin{equation}
\mathrm{R}(x, t)\,+\,\frac{n(n-1)}{\alpha_g}\,
=\,\sum_{P(a,b)}\,\frac{1\,+\,\alpha_g\,\mathcal{K}_{P(a,b)}(x, t)}{\alpha_g}\;.
\end{equation}
Hence, as long as $1\,+\,\alpha_g\,\mathcal{K}_{P(a,b)}(x, t)>\,0$,  we have 
\begin{equation}
\mathrm{R}(x, t)\,+\,\frac{n(n-1)}{\alpha_g}\,>\,0\;,
\end{equation}
 and the  theorem follows.
\end{proof}

\section{AN EXTENDED PERELMAN'S ENERGY FUNCTIONAL}
\label{sect5}
The monotonicity of the Nash entropy functional is a  rather weak  result since it requires that the curvature condition 
$1\,+\,\alpha_g\,\mathcal{K}_{P}(t)\,>\,0$,\,$\forall \,P\in Gr_{(2)}(TM)$, imposed on the initial metric $g$, holds along the evolution of the RG-2 flow, a property that is very difficult to establish.  If we direct our attention to  the behavior of  the Perelman's energy functional \cite{Perelman}
\begin{equation}
\label{orPer}
\mathcal{F}(g(t), {f(t)})\,:=\,\int_M\,\left[\mathrm{R}(g(t))\,+\,|\nabla f(t)|^2_{g(t)} \right]\,d\omega (t)\,=\,
\int_M\, \mathrm{R}^{Per}(g(t))\,d\omega (t)\;,
\end{equation}
the situation, hard to handle in  the standard RG-2 flow (\ref{RG2flow}),  improves considerably along the scale--invariant RG-2 flow $[0, T_0]\,\ni t\,\longmapsto \left(g(t), d{\omega}(t)\,;\,{\xi_{g(t)}}\right)$, defined by (\ref{out1}).  We can exploit the freedom in choosing the drift vector field ${\xi_{g(t)}}$ for controlling the vagaries of the $\mathrm{Rm}^2$ term and obtain monotonicity for a natural variant of  $\mathcal{F}(g(t), {f(t)})$.  We start by recalling the expression of  the pointwise evolution of  Perelman's modified scalar curvature  $\mathrm{R}^{Per}(g(t))$ which appears in (\ref{orPer}). This is indeed instrumental for the characterization of Perelman's $\mathcal{F}$--energy  for the Ricci flow, and plays a  basic role in the RG-2 flow case.\\
\\
\noindent
In full generality, let us consider  the generic (germ of ) curve of metrics $[0,1]\ni t\longrightarrow g(t)$, with tangent vector provided by a smooth ($t$--dependent) symmetric bilinear form $v\in C^{\infty }(M, \otimes ^2_{Sym}T^*M)$,
\begin{equation}
\label{vflow}
\frac{\partial }{\partial t }\,g_{jk}(t )=\,v_{jk}(t)\;.
\end{equation}
Along (\ref{vflow}), we have (for a detailed derivation see  \cite{flowers0}, Chapter 6,   Exercise 6.84, p. 274) 
\begin{eqnarray}
&&\frac{\partial }{\partial t}\,\mathrm{R}^{Per}(g(t))\,=\,\nabla ^j\nabla ^k v_{jk}+v_{jk}\mathrm{R}^{jk}- 2\nabla^jf\nabla^k v_{jk}\nonumber\\
\label{pointwise1}\\
&&+v_{jk}\nabla^j f\nabla^k f + 2\left(\triangle _g-\nabla_k f\nabla^k \right)\,\left(\frac{\partial f}{\partial t}-\frac{1}{2}\,tr_g(v) \right)\nonumber\\
\nonumber\\
&&- 2 v_{jk} (\mathrm{R}^{jk}+\nabla^j \nabla^k)f\;.\nonumber
\end{eqnarray}
It is useful to write the rather complicated expression (\ref{pointwise1}) in terms of the {\it weighted covariant derivative} $\nabla ^{(\omega)}$ associated with the measure $d\omega$ (see (\ref{Wlie})). We extend it to a generic tensor field $T$ over $M$ according to
\begin{equation}
\label{weight}
\nabla ^{(\omega)}\,T\,:=\,e^{f}\,\nabla \,\left( e^{-\,f}\,T  \right)\,=\,\nabla T\,-\,\nabla f\otimes T\;,   
\end{equation}
where $\nabla $ is the Levi-Civita connection on $(M, g, d\omega)$, (or when time-dependent,  $(M, g(t), d\omega(t))$).\, $\nabla ^{(\omega)}$ is a natural differential operator on the Riemannian manifold with density $(M, g,\,d\omega\,=\,e^{-f}d\mu_g)$. 
To rewrite (\ref{pointwise1}) in terms of  $\nabla ^{(\omega)}$, let us apply the easily proven relations
\begin{eqnarray}
&&\nabla _j^{(\omega)}\nabla _k^{(\omega)}\,v^{jk}\,:=\,e^f\nabla _j\left[e^{\,-f}e^f\,\nabla _k\left(e^{\,-f}v^{jk} \right) \right]=
e^f\nabla _j\nabla _k\left(e^{\,-f}v^{jk} \right)\nonumber\\
\\
&&=\,   \nabla _j\nabla _k v^{jk}- 2\nabla _j f\nabla_k v^{jk}+v^{jk}\nabla_j f\nabla_k f
-v^{jk} \nabla_j \nabla_k f \nonumber\;.
\end{eqnarray}
Then
\begin{equation}
 \nabla _j\nabla _k v^{jk}- 2\nabla _j f\nabla_k v^{jk}+v^{jk}\nabla_j f\nabla_k f \,
 =\,\nabla _j^{(\omega)}\nabla _k^{(\omega)}\,v^{jk}\,+\,v^{jk} \nabla_j \nabla_k f \;.
\end{equation}
By introducing this latter expression in (\ref{pointwise1}), and recalling that  $\triangle _g-\nabla_k f\nabla_k\,:=\,\triangle _g^{(\omega)}$, we 
eventually get
\begin{equation}
\label{Perlequation0}
\frac{\partial }{\partial t}\,\mathrm{R}^{Per}(g(t))\,=\,\nabla ^{(\omega)}_j\nabla ^{(\omega)}_k\,v^{jk}\,-\,\mathrm{R}^{BE}_{jk}\,v^{jk}+\,2\,\triangle _g^{(\omega)}\,\left(\frac{\partial f}{\partial t}-\frac{1}{2}\,tr_g(v) \right)\;,
\end{equation}
where
\begin{equation}
\label{BEM}
\mathrm{Ric}^{BE}(g)\,:=\,\mathrm{Ric}(g)\,+\,\nabla\nabla\,f\;,
\end{equation}
denotes the Bakry--Emery Ricci tensor  associated with $(M, g(t), d\omega(t))$.\\
\\
Having dispensed with these preliminary remarks, let us consider the scale--invariant  
RG-2 flow (\ref{out1}) for which, according to Theorem \ref{ThA}, we have short time existence on some interval $t\,\in\,[0, T)$.  \,Let $T_0\,<\,T$, \, and consider the corresponding  time--reversed flow $[0, T_0]\,\ni \eta\,\longmapsto\, g(\eta)$,\,
$\eta\,:=\, T_0\,-\,t$. Along the backward RG-2 flow  $[0, T_0]\,\ni \,\eta\,\longmapsto \,g(\eta)$ so defined  we choose the gradient\footnote{As usual, in what follows we adopt the convention that   $\nabla$,  when acting on a time--dependent vector or tensor field, denotes the covariant derivative with respect to $(M, g(t))$.}   vector field $\eta\,\longmapsto \,{\xi_{g(\eta)}}$ by requiring that  it evolves, starting from  a given initial 
condition  ${\xi_{\eta=0}\,:=\,\nabla\,\psi(\eta=0)}$  according to the non-linear parabolic  PDE
equation 
\begin{equation}
\label{xiFP}
\frac{\partial}{\partial \eta}\,{\xi_{g(\eta)}}\, =\,\triangle _{g(\eta)}\,{\xi_{g(\eta)}}\,-\,\xi_{g(\eta)}\ast \left(\mathrm{Ric}(\eta)+\frac{\alpha_{g}}{4}\,\mathrm{Rm}^2(g(\eta), \xi_{g(\eta)}) \right)\,-\,
\frac{\alpha_g^2}{64}\,\left|\mathrm{Rm}^2(g(\eta), \xi_{g(\eta)})\right|^2_{{g}(\eta)}\,\xi_{g(\eta)}\;,
\end{equation}
where we have introduced the drift--modified squared curvature:
\begin{equation}
\label{RMBE}
\alpha_g\,\mathrm{Rm}^2\left({g}(\eta),\, \xi_{g(\eta)}\right)\,:=\,\alpha_g\,
 \mathrm{Rm}^2({g}(\eta))\,-\,{2}\,\mathrm{L}_{\xi_{g(\eta)}}{g}(\eta)\;,
 \end{equation}
and where the components of the vector endomorphism $\xi_{g}\ast \left(\mathrm{Ric}+\frac{\alpha_{g}}{4}\,\mathrm{Rm}^2(g, \xi_{g}) \right)$ are defined by  $\xi^i_{g}g^{hk}\left(\mathrm{Ric}_{ih}+\frac{\alpha_{g}}{4}\,\mathrm{Rm}^2_{ih}(g, \xi_{g}) \right)$. Finally,  the term
${\alpha_g^2}\,\left|\mathrm{Rm}^2(g(t), \xi_{g(t)})\right|^2_{{g}(t)}$ denotes the squared norm of  $\alpha_g\,\mathrm{Rm}^2(g(t), \xi_{g(t)})$.\\
\\
For a given initial condition, (\ref{xiFP}) is a parabolic PDE  which admits a unique solution $\xi_{g(\eta)}$ for $\eta\,\in\,[0, T_0]$. Notice that along the original $t$ evolution, $[0, T_0]\,\ni\,t\,\longmapsto\, g(t)$ , of the RG-2 flow, (\ref{xiFP}) can be rewritten as the backward parabolic evolution for  $\xi_{g(t)}\,=\,\xi_{g(\eta=T_0-t)}$,\,$t\,\in\,[0, T_0]$ given by
\begin{equation}
\label{xiFPtime}
\frac{\partial}{\partial t}\,{\xi_{g(t)}}\, =\,-\,\triangle _{g(t)}\,{\xi_{g(t)}}\,+\,\xi_{g(t)}\ast \left(\mathrm{Ric}(t)+\frac{\alpha_{g}}{4}\,\mathrm{Rm}^2(g(t), \xi_{g(t)}) \right)\,+\,
\frac{\alpha_g^2}{64}\,\left|\mathrm{Rm}^2(g(t), \xi_{g(t)})\right|^2_{{g}(t)}\,\xi_{g(t)}\;.
\end{equation}
According to Theorem  \ref{ThA}, we can associate to the  time--reversed flows $[0, T_0]\,\ni \eta\,\longmapsto\, \left(g(\eta),\,{\xi_g}(\eta)\right)$,\, 
$\eta\,:=\, T_0\,-\,t$ so defined also the corresponding parabolic  equation  (\ref{out1b})
\begin{equation}
\frac{\partial}{\partial \eta}\, d\omega(\eta)\,=\,\triangle _{g(\eta)}\,d\omega(\eta)\,+\,\mathrm{div}^{\omega}{\xi_g}(\eta)\,d\omega(\eta)\;,\;\;\;\;\;
d\omega(\eta=0)\,=\,d\omega_{(0)}\;,
\end{equation}
whose solution defines the evolution  $\eta\,\longmapsto \,d\omega(\eta)$. This forward/backward parabolic see--saw game characterizes 
the flow $t\,\longmapsto \,\left(d\omega(t=T_0-\eta)\,=\,e^{-\,f(t)}\,d\mu_{g(t)}\,;\,{\xi_{g(t)}}\right)$ to which we can associate the vector field 
\begin{equation}
W(t)\,:=\,-\,\left(\nabla\,f(t) \,-\,{\xi_{g(t)}}\right), \, \;\, t\,\in[0,T_0]\;.
\end{equation}
The next step is to consider the action on (\ref{out1}), (\ref{out1b}), and (\ref{xiFP}) of  the family of diffeomorphisms $[0, T_0]\ni \, t\longmapsto \phi_t$, which solve the non--autonomous system of ODE
\begin{equation}
\label{Wdiff}
\frac{\partial}{\partial t}\,\phi_t(p)\,=\,W\left( \phi_t(p), t \right)\;,\;\;\;\phi_{t=0}\,=\,id_M\;.
\end{equation}
We prove the following results:

\begin{theorem}  
\label{ThBB}
If we denote by $\overline{g}(t)\,:=\,\phi_t^*(g(t))$,\; $d\,\overline{\omega}(t)\,:=\,\phi_t^*(d\omega(t))$, \; $\overline{\xi}_{g(t)}\,:=\,\phi_t^*({\xi_{g(t)}})$,  and 
 $\overline\nabla\,:=\,\nabla_{\overline{g}}$ the relevant pullbacks under the action of the one--parameter family of diffeomorphisms   
$\phi_t$ solving  (\ref{Wdiff}), then the corresponding modified  scale--invariant RG-2 flow associated to the action of  $\phi_t$ on (\ref{out1}), (\ref{out1b}) and  (\ref{xiFP}) is provided by
\begin{eqnarray} 
\frac{\partial }{\partial t }\,\overline{g}(t)\,&=&\,-\,
2\mathrm{Ric}^{BE}(\overline{g}(t ))\,-\,\frac{\alpha_g}{2}\,\mathrm{Rm}^2(\overline{g}(t), \overline{\xi}_{g(t)}) \;, \nonumber\\
 \nonumber\\
 \frac{\partial }{\partial t }\,d\overline{\omega}(t)\,&=&\,0\;,\label{DToutPer}\\
 \nonumber\\
 \ppt  \overline{\xi}_g(t)\,&=&\,-\,
\triangle _{\overline{g}(t)}^{(\omega)}\,{\overline{\xi}_g(t)}\,+\,\overline{\xi}_g(t)\ast \left(\mathrm{Ric}^{BE}(\overline{g}(t))+\frac{\alpha_{g}}{4}\,\mathrm{Rm}^2(\overline{g}(t), \, {\overline{\xi}_g(t)}) \right)\nonumber\\
&&\,+\, 
\frac{\alpha_g^2}{64}\,\left|\mathrm{Rm}^2(\overline{g}(t), {\overline{\xi}_g(t)})\right|^2_{\overline{g}(t)}\,\overline{\xi}_g(t) \;, \nonumber
\end{eqnarray}
and if along $t\,\longmapsto \,\left( M, \overline{g}(t), d\,\overline{\omega}(t),\,\overline{\xi}_{g(t)} \right)$,  we define the extended Perelman energy according to 
\begin{eqnarray}
\label{orPerextended0}
\mathcal{F}(\overline{g}(t), {\overline{f}(t)}, \overline{\xi}_g(t) )\,&:=&\,\int_M\,\left[\mathrm{R}(\overline{g}(t))\,+\,|\overline{\nabla}\, \overline{f}(t)|^2_{\overline{g}(t)}\,+\,|\overline{\nabla}\, \overline{\psi}(t)|^2_{\overline{g}(t)} \right]\,d\overline{\omega} (t)\\
&=&\,
\int_M\, \left(\mathrm{R}^{Per}(\overline{g}(t))\,+\,\left|\overline{\xi}_{g}(t)\right|^2\right)\,d\omega (t)\;,\nonumber
\end{eqnarray}
then we have

\begin{equation}
\label{PerRG2F} 
\frac{d}{dt}\,\mathcal{F}(\overline{g}(t), {\overline{f}(t)}, \overline{\xi}_g(t))
\,\geq\,\,2\,\int_M\,\left| \mathrm{Ric}^{BE}(\overline{g}(t))\,+\,\frac{\alpha_g}{8}\mathrm{Rm}^2(\overline{g}(t), \overline{\xi}_{g(t)})\right|^2_{\overline{g}(t)}\,d\overline{\omega} (t)\;.
\end{equation}
\end{theorem}

As a rather direct consequence of this result we also have 
\begin{theorem} 
\label{ThBB2} (The extended Perelman $\Lambda(g)$--functional).
Given a metric $\overline{g}$ on the closed manifold $M$, let $\mathcal{F}(\overline{g}, {\overline{f}}, \overline{\xi}_g)$ be the extended Perelman energy associated to  the generic  $\overline{f}\,\in\,C^\infty(M, \mathbb{R})$ and  gradient vector field $\overline{\xi}_g=\,\overline{\nabla}\,\overline{\psi}$. Set $\overline{h}\,:=\,e^{-\,\overline{f}/2}$ and define $\mathcal{F}(\overline{g}, \overline{h}, \overline{\psi})$ according to
\begin{equation}
\mathcal{F}(\overline{g}, \overline{h}, \overline{\psi})\,:=\,
\int_M\,\left[\left(\mathrm{R}(\overline{g})\,+\,|\overline{\nabla}\, \overline{\psi}|^2\right)\,\overline{h}^2\,+\,
4\,|\overline{\nabla}\, \overline{h}|^2 \right]\,d\mu_{\overline{g}}\;.
\end{equation}
We let $\overline{h}$ and $\overline{\psi}$ both be in $\mathrm{W}^{1,2}(M)$, the (completion of the) Sobolev space of $C^\infty$ functions with finite $\mathrm{W}^{1,2}$ norm with respect to the Riemannian measure 
$d\mu_{\overline{g}}$, and denote by 
\begin{equation}
\Gamma\,:=\,\left\{\left(\overline{h},\,\overline{\psi}\right)\,\in\,\mathrm{W}^{1,2}(M)\times \mathrm{W}^{1,2}(M)\,:\,\int_M\,\overline{\psi}\;\overline{h}^2\,d\mu_{\overline{g}}\,=\,0,\;\int_M\,\overline{\psi}^2\;\overline{h}^2\,d\mu_{\overline{g}}\,=\,(\alpha_g)^{n/2}\,=\,\int_M\,\overline{h}^2\,d\mu_{\overline{g}}\right\}\;,
\end{equation}
the variational set in $\mathrm{W}^{1,2}(M)\times \mathrm{W}^{1,2}(M)$ over which the functional 
$(\overline{h},\,\overline{\psi})\,\longmapsto \,\mathcal{F}(\overline{g}, \overline{h}, \overline{\psi})$ is minimized. If we define
\begin{equation}
\label{Lambdamanip}
\Lambda[\overline{g}]\,:=\,\inf_{\Gamma}\,\left\{\mathcal{F}(\overline{g}, \overline{h}, \overline{\psi}) \right\}\,=\,\inf_{\Gamma}\,\int_M\,\left[\left(\mathrm{R}(\overline{g})\,+\,|\overline{\nabla}\, \overline{\psi}|^2\right)\,\overline{h}^2\,+\,
4\,|\overline{\nabla}\, \overline{h}|^2 \right]\,d\mu_{\overline{g}}\;,
\end{equation}
then on the Riemannian manifold $(M,g)$ there exists a pair of real numbers $\left(\lambda_1[\overline{g}],\,\lambda_2[\overline{g}]\right)$ and  a unique pair  $\left(\overline{h}_0,\,\overline{\psi}_0\right)\,\in\,\Gamma$ that is a solution of the coupled elliptic eigenvalue problem 
\begin{eqnarray}
-\,4\,\Delta_{\overline{g}}\,\overline{h}_0\,+\,\left(\mathrm{R}(\overline{g})\,+\, 
|\overline{\nabla}\, \overline{\psi}_0|^2\,-\, \lambda_2[\overline{g}]\, \overline{\psi}_0^2  \right)\,\overline{h}_0&=&\lambda_1[\overline{g}]\,\overline{h}_0\;,\nonumber\\
\label{PDEsystem*}\\
-\,\Delta_{\overline{g}}\,\overline{\psi}_0\,-\,\overline{\nabla}_i\,\ln\overline{h}^2_0\,\overline{\nabla}^i\,\overline{\psi}_0&=&
\lambda_2[\overline{g}]\,\overline{\psi}_0\nonumber\;,
\end{eqnarray}
such that for every pair $\left(\overline{h},\,\overline{\psi}\right)\,\in\,\Gamma$ we have
\begin{equation}
\label{infFeigen*}
\int_M\,\left[\left(\mathrm{R}(\overline{g})\,+\,|\overline{\nabla}\, \overline{\psi}|^2\right)\,\overline{h}^2\,+\,
4\,|\overline{\nabla}\, \overline{h}|^2 \right]\,d\mu_{\overline{g}}\,\geq\,(\alpha_g)^{n/2}\left(\lambda_1[\overline{g}]\,+\,\lambda_2[\overline{g}]\right)\;,
\end{equation}
with equality for  $\left(\overline{h},\,\overline{\psi}\right)\,=\,\left(\overline{h}_0,\,\overline{\psi}_0\right)$. 
 Moreover, if we assume that the Bakry--Emery Ricci curvature (\ref{BEM}) of  $(M, \overline{g}, e^{\,-\,\overline{f}_0}\,d\mu_{\overline{g}})$ is bounded below according to 
\begin{equation}
\label{BEM**}
\mathrm{Ric}^{BE}(\overline{g})\,:=\,\mathrm{Ric}(\overline{g})\,+\,\overline{\nabla}\,\overline{\nabla}\,\overline{f}_0\,\geq\,\mathrm{C}_0\,\overline{g}\;,
\end{equation}
for some constant $\mathrm{C}_0\,\in\,\mathbb{R}$, then 
\begin{equation}
\Lambda[\overline{g}]
\,=\,(\alpha_g)^{n/2}\left(\lambda_1[\overline{g}]\,+\,\lambda_2[\overline{g}]\right)\,\geq\,
(\alpha_g)^{n/2}\lambda(\overline{g})\,+\,\sup_{s\in (0,1)}\,\left\{4s(1-s)\,\frac{\pi^2}{\mathrm{diam}^2(\overline{g})}\,+\,s\,\mathrm{C}_0   \right\}\;,
\end{equation}
where $\mathrm{diam}^2(\overline{g})$ is the diameter of  $(M, \overline{g})$ and where
$\lambda(\overline{g})$ is Perelman's $\lambda$--functional,
\begin{equation}
\label{Perelambda}
\lambda(\overline{g})\,:=\,\inf_{\overline{f}}\,\left\{\int_M\,\mathrm{R}^{Per}(\overline{g}(t))\,e^{\,-\,\overline{f}}\,d\mu_{\overline{g}}\,:\,\overline{f}\in C^\infty(M,\mathbb{R}),\,\int_M\,e^{\,-\,\overline{f}}\,d\mu_{\overline{g}}\,=\,(\alpha_g)^{n/2} \right\}\;.
\end{equation} 
Finally,
for all $t\,>\,0$ for which the RG-2 flow exists
\begin{equation}
 \frac{d}{d t_-}\,\Lambda[\overline{g}(t)]\,\geq\,\,2\,\int_M\,\left| \mathrm{Ric}^{BE}(\overline{g}(t))\,+\,\frac{\alpha_g}{8}\mathrm{Rm}^2(\overline{g}(t), \overline{\xi}_{g(t)})\right|^2_{\overline{g}(t)}\,d\overline{\omega} (t)\;,
 \end{equation}
 where  the time derivative is in the sense of the $\liminf$ of backward difference quotients.
\end{theorem}
\begin{proof} (Of Theorem \ref{ThBB}).\,
 Since $M$ is compact, (\ref{Wdiff}) defines a one--parameter family of diffeomorphisms  as long as the solutions of  (\ref{out1}), (\ref{out1b}), and (\ref{xiFP})  exist; in particular, we may assume that $\left\{\phi_t\,\in\,\mathrm{Diff}(M)\,|\,t\,\in\,[0,T_0] \right\}$,  and we consider the  relevant pullbacks   $\overline{g}(t)\,:=\,\phi_t^*(g(t))$,\; $d\,\overline{\omega}(t)\,:=\,\phi_t^*(d\omega(t))$,\,   
 $\overline\nabla\,:=\,\nabla_{\overline{g}}$, and $\overline{{\xi_g}}(t)\,:=\,\phi_t^*({\xi_{g(t)}})$ . In the latter, we used the notation  
$\phi_t^*({\xi_{g(t)}})\,:=\,(\phi_t)^{- 1}_*({\xi_{g(t)}})$. Starting with $\phi_t^*(d\omega(t))$, we have 

\begin{align}
\label{ficomput1}
\ppt {\phi_t}^*( d\omega(t)) &= \frac{\partial}{\partial s} | _{s=0} (\phi_{t+s}^* d\omega(t+s)\nonumber \\
&= {\phi_t}^* (\frac{\partial}{\partial t} d\omega(t)) +  \frac{\partial}{\partial s} | _{s=0}({\phi_{t+s}}^* d\omega(t)) \nonumber\\
&={\phi_t}^*\left(-\,\triangle _{g(t)}\,d\omega(t)\,-\,\mathrm{div}^{\omega}{\xi_{g(t)}}\,d\omega(t)   \right) + \mathrm{L}_{(\phi_t^{-1})_* W(t)} ( {\phi_t}^* d\omega(t)).
\end{align}
As usual, we calculate 
\begin{align*}
\mathrm{L}_{(\phi_t^{-1})_* W(t)} ( {\phi_t}^* d\omega(t)) &=  {\phi_t}^* (\mathrm{L}_{W(t)} d\omega(t)) \\
&=  {\phi_t}^* \left(\mathrm{L}_{\xi_g}\, d\omega(t)\,-\,\mathrm{L}_{\nabla f}\,d\omega(t)\right)\\
&={\phi_t}^* \left(\mathrm{div}^{(\omega)}{\xi_{g(t)}}\,d\omega(t)\,-\,\mathrm{div}^{(\omega)}\nabla f(t)\,d\omega(t)\right)\\
&={\phi_t}^* \left(\mathrm{div}^{(\omega)}{\xi_{g(t)}}\,d\omega(t)\,-\,\triangle ^{(\omega)} f(t)\,d\omega(t)\right)\\
&={\phi_t}^* \left(\mathrm{div}^{(\omega)}{\xi_{g(t)}}\,d\omega(t)\,+\,\triangle_{g(t)}\,d\omega(t)\right)\;,
\end{align*} 
where we have used the characterization (\ref{Wlie}) of the weighted divergence $\mathrm{div}^{(\omega)}$ in terms of the Lie derivative of the measure  $d\omega(t)$, and the relation $\Delta_g \,d\omega\,=\,-\,\Delta _g^{(\omega)} f\,d\omega$ (see (\ref{Wlap})).  Introducing this result into 
(\ref{ficomput1}), we get
\begin{equation}
\ppt d\overline{\omega}(t)\,:=\, \ppt {\phi_t}^*( d\omega(t))\,=\,0\;.
\end{equation}
Similarly, from (\ref{xiFP})  we compute 

\begin{align}
&\ppt ({\phi_t}^* ({\xi_{g(t)}}))\,= {\phi_t}^* \left(\frac{\partial}{\partial t} ({\xi_{g(t)}})) +  \frac{\partial}{\partial s} | _{s=0}({\phi_{t+s}}^* ({\xi_{g(t)}})\right) \\
&={\phi_t}^*\left(-\,\triangle _{g(t)}\,{\xi_{g(t)}}\,+\,\xi_{g(t)}\ast \left(\mathrm{Ric}(t)+\frac{\alpha_{g}}{4}\,\mathrm{Rm}^2(g(t), {\xi_{g(t)}}) \right)\right.\nonumber\\
&+\, 
\frac{\alpha_g^2}{64}\,\left.\left|\mathrm{Rm}^2(g(t), {\xi_{g(t)}})\right|^2_{{g}(t)}\,\xi_{g(t)}\, + \mathrm{L}_W ({\xi_{g(t)}})  \right)\;.\nonumber
\end{align}
Since $\mathrm{L}_{\xi_{g(t)}}({\xi_{g(t)}})\,\equiv\, 0$, we have

\begin{equation}
\mathrm{L}_W (\xi_{g(t)})\,=\, \mathrm{L}_{\xi_{g(t)}} (\xi_{g(t)})\,-\,\mathrm{L}_{\nabla\,f} (\xi_{g(t)})\,=\,-\,\mathrm{L}_{\nabla\,f} (\xi_{g(t)})\;,
\end{equation}
which in local coordinates reads
\begin{equation}
\left(\mathrm{L}_W ({\xi_{g(t)}})\right)^k\,=\,\xi^i_{g(t)}\,\nabla_i\nabla^k\,f(t)\,-\,\nabla^i\,f(t)\,\nabla_i\xi^k_{g(t)}\;.
\end{equation}
If we take into account these relations, the definition (\ref{BEM}) of the Bakry--Emery Ricci tensor and  the characterization of the weighted Laplacian
\begin{equation}
\triangle _{g(t)}\,\left({\xi_{g(t)}} \right)\,-\,\nabla^k f(t)\,\nabla_k\left({\xi_{g(t)}}\right)\,=\,
\triangle _{g(t)}^{(\omega)}\,\left({\xi_{g(t)}} \right)\;,
\end{equation}
we eventually get for the evolution of $\overline{\xi}_g(t)\,:=\,{\phi_t}^* ({\xi_{g(t)}})$ the expression
\begin{eqnarray}
\ppt  \overline{\xi}_g(t)\,&=&\,-\,
\triangle _{\overline{g}(t)}^{(\omega)}\,{\overline{\xi}_g(t)}\,+\,\overline{\xi}_g(t)\ast \left(\mathrm{Ric}^{BE}(\overline{g}(t))+\frac{\alpha_{g}}{4}\,\mathrm{Rm}^2(\overline{g}(t), \, {\overline{\xi}_g(t)}) \right)\\
\,&+&\, 
\frac{\alpha_g^2}{64}\,\left|\mathrm{Rm}^2(\overline{g}(t), {\overline{\xi}_g(t)})\right|^2_{\overline{g}(t)}\,\overline{\xi}_g(t) \;.\nonumber
\end{eqnarray}
Finally, for the pulled--back metric we have the standard DeTurck computation
\begin{align}
\label{gcomput1}
\frac{\partial }{\partial t }\bar{g}\,:=\,
\ppt {\phi_t}^*( g(t)) &= \frac{\partial}{\partial s} | _{s=0} (\phi_{t+s}^* g(t+s)\nonumber \\
&= {\phi_t}^* (\frac{\partial}{\partial t} g(t)) +  \frac{\partial}{\partial s} | _{s=0}({\phi_{t+s}}^* g(t)) \nonumber\\
&={\phi_t}^*\left(-\,2\mathrm{Ric}(t)\,-\,\frac{\alpha_g}{2}\mathrm{Rm}^{2}\right) + \mathrm{L}_{(\phi_t^{-1})_* W(t)} ( {\phi_t}^* g(t))\nonumber\\
&= -2 \mathrm{Ric}(\bar{g}(t)) - \frac{\alpha_g}{2} \mathrm{Rm}^2(\bar{g}(t))-2\overline{\nabla}\, \overline{\nabla}\, \overline{f}\,+\,\mathrm{L}_{\overline{\xi}_g}\overline{g}(t)\nonumber\;. 
\end{align} 

Putting these results together, we find that the flow  $[0, T_0]\,\ni \,t\,\longmapsto\, (\bar g(t), d\overline{\omega}(t), \overline{\xi}_{g(t)})$ is a solution of the following system:
\begin{eqnarray} 
&&\frac{\partial }{\partial t }\bar{g}\,=\, -2 \mathrm{Ric}^{BE}(\bar{g}(t)) - \frac{\alpha_g}{2} \mathrm{Rm}^2(\bar{g}(t), \overline{\xi}_{g(t)}) \label{gevol} \\
 &&\frac{\partial}{\partial t}\,d\overline{\omega}(t)\,=\,0 \label{fevol}\\
&&\ppt  \overline{\xi}_g(t)\,=\,-\,
\triangle _{\overline{g}(t)}^{(\omega)}\,{\overline{\xi}_g(t)}\,+\,\overline{\xi}_g(t)\ast \left(\mathrm{Ric}^{BE}(\overline{g}(t))+\frac{\alpha_{g}}{4}\,\mathrm{Rm}^2(\overline{g}(t), \, {\overline{\xi}_g(t)}) \right)\label{wabevol}\\
&&\,+\, 
\frac{\alpha_g^2}{64}\,\left|\mathrm{Rm}^2(\overline{g}(t), {\overline{\xi}_g(t)})\right|^2_{\overline{g}(t)}\,\overline{\xi}_g(t) \;.
\nonumber
\end{eqnarray}
as stated. Here $\mathrm{Ric}^{BE}$ is the Bakry--Emery Ricci tensor (\ref{BEM}) associated with the Riemannian manifold with density $(M, \bar{g}(t), d\overline{\omega}(t))$  and  $\alpha_g\,\mathrm{Rm}^2(\overline{g}(t),\,\overline{\xi}_{g(t)})$ is
the corresponding short-hand notation (\ref{RMBE})   for the drift-modified squared curvature. \\
\\
Along the flow  $[0, T_0]\,\ni \,t\,\longmapsto\, (\bar g(t), d\overline{\omega}(t), \overline{\xi}_{g(t)})$ (\ref{gevol}),   (\ref{fevol}), and 
 (\ref{wabevol}),  we compute
 \begin{eqnarray}
 \label{xisquared}
 &&\ppt  \left|\overline{\xi}_g(t)\right|^2\,:=\,\ppt \,\left(\overline{g}_{ik}(t)\, \overline{\xi}^i_g(t)\overline{\xi}^k_g(t)   \right)\,=\,
 2\,\left(\overline{g}_{ik}(t)\, \overline{\xi}^i_g(t)\ppt\overline{\xi}^k_g(t)   \right)\,+\,\overline{\xi}^i_g(t)\overline{\xi}^k_g(t)\,\ppt 
 \overline{g}_{ik}(t)\\
 &&=\,-\,2 \overline{g}_{ik}(t)\, \overline{\xi}^i_g(t)  \triangle _{\overline{g}(t)}^{(\omega)}\,{\overline{\xi}^k_g(t)}\,
 \,+\,2\overline{\xi}^i_g(t) \overline{\xi}^k_g(t) \left(\mathrm{Ric}^{BE}(\overline{g}(t))+\frac{\alpha_{g}}{4}\,\mathrm{Rm}^2(\overline{g}(t), \, {\overline{\xi}_g(t)}) \right)_{ik}\nonumber\\
&&\,+\, 
\frac{\alpha_g^2}{32}\,\left|\mathrm{Rm}^2(\overline{g}(t), {\overline{\xi}_g(t)})\right|^2_{\overline{g}(t)}\,\left|\overline{\xi}_g(t)\right|^2 
\,-\,2\overline{\xi}^i_g(t) \overline{\xi}^k_g(t) \left(\mathrm{Ric}^{BE}(\overline{g}(t))+\frac{\alpha_{g}}{4}\,\mathrm{Rm}^2(\overline{g}(t), \, {\overline{\xi}_g(t)}) \right)_{ik}\nonumber\\
&&=\,-\,\triangle _{\overline{g}(t)}^{(\omega)}\,\left|{\overline{\xi}_g(t)}\right|^2\,+\,2\left| \overline{\nabla}\, {\overline{\xi}_g(t)} \right|^2
\,+\, 
\frac{\alpha_g^2}{32}\,\left|\mathrm{Rm}^2(\overline{g}(t), {\overline{\xi}_g(t)})\right|^2_{\overline{g}(t)}\,\left|\overline{\xi}_g(t)\right|^2 \nonumber\;,
 \end{eqnarray}
 where we exploited the  relation $\triangle _{\overline{g}(t)}^{(\omega)}\,\left|{\overline{\xi}_g(t)}\right|^2\,=\,2 \overline{g}_{ik}(t)\, \overline{\xi}^i_g(t)  \triangle _{\overline{g}(t)}^{(\omega)}\,{\overline{\xi}^k_g(t)}+2\left| \overline{\nabla}\, {\overline{\xi}_g(t)} \right|^2$, immediate consequence of the analogous relation which holds for the Laplacian. In terms of the backward time $\eta\,\in\,[0, T_0]$, governing the forward parabolic evolution (\ref{xiFP}) of the drift vector field $\xi_{g(\eta)}$, (\ref{xisquared}) can be written as the scalar parabolic PDE with non--positive nonlinear reaction terms
 \begin{equation}
 \frac{\partial}{\partial \eta}  \left|\overline{\xi}_g(\eta)\right|^2\,=\,\,\triangle _{\overline{g}(\eta)}^{(\omega)}\,\left|{\overline{\xi}_g(\eta)}\right|^2\,-\,2\left| \overline{\nabla}\, {\overline{\xi}_g(\eta)} \right|^2
\,-\, 
\frac{\alpha_g^2}{32}\,\left|\mathrm{Rm}^2(\overline{g}(\eta), {\overline{\xi}_g(\eta)})\right|^2_{\overline{g}(\eta)}\,\left|\overline{\xi}_g(\eta)\right|^2 \;.
 \end{equation}
 From the parabolic maximum principle it immediately follows that
 along the flow  $[0, T_0]\,\ni \,\eta\,\longmapsto\, (\bar g(\eta), d\overline{\omega}(\eta), \overline{\xi}_{g(\eta)})$  the squared norm  $\left|\overline{\xi}_g(\eta)\right|^2$ is non--increasing, \, \textit{i.e.}, $\left|\overline{\xi}_g(\eta)\right|^2\,\leq\,\left|\overline{\xi}_g(\eta=0)\right|^2$, which, in terms of the $t$-evolution of the drift vector field, implies   
 \begin{equation}
 \label{Tincreasing}
 \left|\overline{\xi}_g(t)\right|^2\,\geq\,\left|\overline{\xi}_g(t=0)\right|^2\,\geq\, 1\;,\;\;\;\;\;\; t\,\in\,[0, T_0]\;,
 \end{equation}
 where, by rescaling $\overline{\xi}_g(\eta=0)$ if necessary, we have assumed, without loss in generality, that   $\max_{x\in M}\left|\overline{\xi}_g(t=0)\right|^2\,=\,1$.  Moreover, by integrating (\ref{xisquared}) over $(M, d\omega)$, the preservation of the measure $d\overline{\omega}(t)$  along   $[0, T_0]\,\ni \,t\,\longmapsto\, (\bar g(t), d\overline{\omega}(t), \overline{\xi}_{g(t)})$  implies   that
 \begin{eqnarray}
 \label{Lowbound}
\frac{d}{dt}\,\int_M\,\left|\overline{\xi}_g(t)\right|^2\,d\overline{\omega}(t)\,&=&\,\int_M\, \left( 2\left| \overline{\nabla}\, {\overline{\xi}_g(t)} \right|^2
\,+\, 
\frac{\alpha_g^2}{32}\,\left|\mathrm{Rm}^2(\overline{g}(t), {\overline{\xi}_g(t)})\right|^2_{\overline{g}(t)}\,\left|\overline{\xi}_g(t)\right|^2\right)\, d\overline{\omega}(t)\\
&\geq &\,\frac{\alpha_g^2}{32}\,\int_M\, 
\left|\mathrm{Rm}^2(\overline{g}(t), {\overline{\xi}_g(t)})\right|^2_{\overline{g}(t)}\, d\overline{\omega}(t)\;,\nonumber
 \end{eqnarray}
 where we have exploited the pointwise bound (\ref{Tincreasing}).  \\
 \\
 Along the  flows $[0, T_0]\,\ni \,t\,\longmapsto\, (\bar g(t), d\overline{\omega}(t), \overline{\xi}_{g(t)})$ so defined, it is rather natural to extend Perelman's energy functional  (\ref{orPer}) according to
 
 \begin{eqnarray}
\label{orPerextended}
\mathcal{F}(\overline{g}(t), {\overline{f}(t)}, \overline{\xi}_g(t) )\,&:=&\,\int_M\,\left[\mathrm{R}(\overline{g}(t))\,+\,|\overline{\nabla}\, \overline{f}(t)|^2_{\overline{g}(t)}\,+\,|\overline{\nabla}\, \overline{\psi}(t)|^2_{\overline{g}(t)} \right]\,d\overline{\omega} (t)\\
&=&\,
\int_M\, \left(\mathrm{R}^{Per}(\overline{g}(t))\,+\,\left|\overline{\xi}_{g}(t)\right|^2\right)\,d\omega (t)\;.
\end{eqnarray}
 To discuss the monotonicity properties of this extension let start by noticing that along  $[0, T_0]\,\ni \,t\,\longmapsto\, (\bar g(t), d\overline{\omega}(t), \overline{\xi}_{g(t)})$ the pointwise evolution (\ref{Perlequation0}) of the Perelman modified scalar curvature reduces to
 
 \begin{equation}
\label{Perlequation1RG2}
\frac{\partial }{\partial t}\,\mathrm{R}^{Per}(\overline{g}(t))\,=\,\overline{\nabla} ^{(\omega)}_j\overline{\nabla} ^{(\omega)}_k\,\overline{RG}_{jk}(t)\,-\,\mathrm{R}^{BE}_{jk}(\overline{g}(t))\,\overline{RG}_{jk}(t)\;,
\end{equation}
where 
\begin{equation}
\overline{RG}_{jk}(t)\,:=\,  -2 \mathrm{Ric}^{BE}(\bar{g}(t)) - \frac{\alpha}{2} \mathrm{R}m^2(\bar{g}(t), \overline{\xi}_{g(t)}) 
\end{equation}
is the generator of the RG-2 flow (\ref{gevol}), and where the measure--variation term $2\,\triangle _g^{(\omega)}\,\left(\frac{\partial f}{\partial t}-\frac{1}{2}\,tr_g(v) \right)$ present in the expression (\ref{Perlequation0}) vanishes because the pulled--back measure $d\overline{\omega}(t)$ is preserved along the evolution  (\ref{gevol}),   (\ref{fevol}), (\ref{wabevol}).
Explicitly, let us recall that along a generic (germ of) curve of metrics $[0,1]\,\ni \,t\,\longmapsto \,g(t)$ with tangent vector $v\in C^\infty(M, \otimes ^2_{sym}T^*M)$
\begin{equation}
\label{tangentv**}
\frac{\partial}{\partial t}\,g_{jk}(t)\,=\,v_{jk}\;,
\end{equation}
we have the standard computation
\begin{equation}
\label{comput1}
\ppt\,d\omega(t)\,=\,\ppt \left(e^{-\,f(t)}\,d\mu_{g(t)}  \right)\,=\,-\,e^{-\,f(t)}\,d\mu_{g(t)}\,\left(\frac{\partial f}{\partial t}-\frac{1}{2}\,tr_g(v) \right)\;,
\end{equation}
from which the connection between $\ppt\,d\omega(t)\,=\,0$  and the  relation $\frac{\partial f}{\partial t}-\frac{1}{2}\,tr_g(v)\,=\,0$  immediately follows. 
This preservation of the measure $d\overline{\omega}(t)$  also implies that we can integrate over $(M, \overline{g}(t), d\overline{\omega}(t)) $ to obtain 
\begin{equation}
\label{perRg1}
\frac{d}{dt}\,\int_M\,\mathrm{R}^{Per}(\overline{g}(t))\,d\overline{\omega}(t)\,=\,
-\,\int_M\,\mathrm{R}^{BE}_{jk}(\overline{g}(t))\,\overline{RG}_{jk}(t)\,d\overline{\omega}(t)\;,
\end{equation}
where we have integrated away the divergence term 
$\overline{\nabla} ^{(\omega)}_j\overline{\nabla} ^{(\omega)}_k\,\overline{RG}_{jk}(t)$. 
We have
 \begin{eqnarray}
\label{perRg2}
&&-\,\int_M\,\mathrm{R}^{BE}_{jk}(\overline{g}(t))\,\overline{RG}_{jk}(t)\,d\overline{\omega}(t)\\
&&=\,2\,\int_M\,\left[\left|\mathrm{Ric}^{BE}(\overline{g}(t))\right|^2\,+\,
\frac{\alpha_g}{4}\,\mathrm{R}^{BE}_{jk}(\overline{g}(t)) \mathrm{Rm}_{jk}^2(\bar{g}(t), \overline{\xi}_{g(t)})\right]\,d\overline{\omega}(t)\;,\nonumber
\end{eqnarray}
which, by completing the square, can be written as

\begin{equation}
 2\,\int_M\,\left| \mathrm{Ric}^{BE}(\overline{g}(t))\,+\,\frac{\alpha_g}{8}\mathrm{Rm}^2(\overline{g}(t), \overline{\xi}_{g(t)})\right|^2_{\overline{g}(t)}\,d\overline{\omega} (t)\,-\, \frac{\alpha^2_g}{32}\,\int_M\,\left|\mathrm{Rm}^2(\overline{g}(t), \overline{\xi}_{g(t)})\right|^2_{\overline{g}(t)}\,d\overline{\omega} (t)\;.
 \label{square}  
\end{equation}
Hence, if we take into account the evolution and the lower bound of  $\int_M\,\left|\overline{\xi}_{g}(t)\right|^2\,d\overline{\omega}(t)$ provided by  (\ref{Lowbound}), we get
\begin{eqnarray}
\label{nearMiss0}
&&\frac{d}{dt}\,\int_M\,\left(\mathrm{R}^{Per}(\overline{g}(t))\,+\, \left|\overline{\xi}_{g}(t)\right|^2   \right)\,d\overline{\omega}(t)\\
&&\,=\,2\,\int_M\,\left| \mathrm{Ric}^{BE}(\overline{g}(t))\,+\,\frac{\alpha_g}{8}\mathrm{Rm}^2(\overline{g}(t), \overline{\xi}_{g(t)})\right|^2_{\overline{g}(t)}\,d\overline{\omega} (t) \,-\, \frac{\alpha^2_g}{32}\,\int_M\,\left|\mathrm{Rm}^2(\overline{g}(t), \overline{\xi}_{g(t)})\right|^2_{\overline{g}(t)}\,d\overline{\omega} (t)\nonumber\\
&&\, +\,
\int_M\, \left( 2\left| \nabla {\overline{\xi}_g(t)} \right|^2
\,+\, 
\frac{\alpha_g^2}{32}\,\left|\mathrm{Rm}^2(\overline{g}(t), {\overline{\xi}_g(t)})\right|^2_{\overline{g}(t)}\,\left|\overline{\xi}_g(t)\right|^2\right)\, d\overline{\omega}(t)\nonumber\\
&&\,\geq\,\,2\,\int_M\,\left| \mathrm{Ric}^{BE}(\overline{g}(t))\,+\,\frac{\alpha_g}{8}\mathrm{Rm}^2(\overline{g}(t), \overline{\xi}_{g(t)})\right|^2_{\overline{g}(t)}\,d\overline{\omega} (t)\;,\nonumber
\end{eqnarray}
from which the theorem immediately follows.
 \end{proof}
 We now exploit this monotonicity result to  prove  Theorem \ref{ThBB2}.  
\begin{proof}  (Theorem \ref{ThBB2}).  
Given the vector field\footnote{At any given instant $t$;\,when working at fixed time $t$ we drop the explicit time dependence in what follows.} $\overline{\xi}_{{g}}$ we can associate with it the corresponding potential $\overline{\psi}$ by solving the elliptic PDE (\ref{ottoPDE}). If we introduce $\overline{\psi}$, then the extended Perelman energy $\mathcal{F}(\overline{g}, {\overline{f}}, \overline{\xi}_g)$ differs from the the standard Perelman energy only by the additive contribution of the Dirichlet term $\int_M\,\left|\overline{\nabla}\,\overline{\psi}\right|^2\,d\overline{\omega}$. 
Hence, if we take the infimum of  $\mathcal{F}(\overline{g}, {\overline{f}}, \overline{\xi}_g)$  over $f$ and $\psi$, we obtain an invariant of  $(M, g)$ which is basically Perelman's $\lambda(g)$ functional (see  \cite{flowers0} Chap. 5 section 3.1)  interacting with the weighted Laplacian $\Delta ^{(\overline{\omega})}$. 
Explicitly, given a metric $\overline{g}$ on the closed manifold $M$, let $\mathcal{F}(\overline{g}, {\overline{f}}, \overline{\xi}_g)$ be the extended Perelman energy (\ref{orPerextended0})  associated to  the generic  $\overline{f}\,\in\,C^\infty(M, \mathbb{R})$ and  gradient vector field $\overline{\xi}_g=\,\overline{\nabla}\,\overline{\psi}$, for $\psi\,\in\,C^\infty(M, \mathbb{R})$.
If we set $\overline{h}\,:=\,e^{-\,\overline{f}/2}$  and take into account the relation $4\,\left|\overline{\nabla}\,\overline{h}\right|^2\,=\,\left|\overline{\nabla}\,\overline{f}\right|^2\,e^{-\,\overline{f}}$ then we can equivalently rewrite the functional $\mathcal{F}(\overline{g}, {\overline{f}}, \overline{\xi}_g)$ in terms of $\overline{h}$ and  $\overline{\psi}$ as
\begin{equation}
\label{Facca}
\mathcal{F}(\overline{g}, \overline{h}, \overline{\psi})\,:=\,
\int_M\,\left[\left(\mathrm{R}(\overline{g})\,+\,|\overline{\nabla}\, \overline{\psi}|^2\right)\,\overline{h}^2\,+\,
4\,|\overline{\nabla}\, \overline{h}|^2 \right]\,d\mu_{\overline{g}}\;.
\end{equation}
We let $\overline{h}$  and $\overline{\psi}$ both be in $\mathrm{W}^{1,2}(M)$, the (completion of the) Sobolev space of $C^\infty$ functions with finite $\mathrm{W}^{1,2}$ norm with respect to the Riemannian measure 
$d\mu_{\overline{g}}$, and consider the variational set
\begin{equation}
\Gamma\,:=\,\left\{\left(\overline{h},\,\overline{\psi}\right)\,\in\,\mathrm{W}^{1,2}(M)\times \mathrm{W}^{1/2}(M)\,:\,\int_M\,\overline{\psi}\;\overline{h}^2\,d\mu_{\overline{g}}\,=\,0,\;\int_M\,\overline{\psi}^2\;\overline{h}^2\,d\mu_{\overline{g}}\,=\,(\alpha_g)^{n/2}\,=\,\int_M\,\overline{h}^2\,d\mu_{\overline{g}}\right\}\;,
\end{equation}
where the condition $\int_M\,\overline{h}^2\,d\mu_{\overline{g}}\,=\,(\alpha_g)^{n/2}$ stems from the relation $\int_M\,d\overline{\omega}\,=\,(\alpha_g)^{n/2}$ connecting the coupling $\alpha_g$ to the measure $d\overline{\omega}=e^{\,-\,\overline{f}}\,d\mu_{\overline{g}}$, \,whereas the normalization  $\int_M\,\overline{\psi}^2\;\overline{h}^2\,d\mu_{\overline{g}}\,=\,(\alpha_g)^{n/2}$ is chosen for later convenience. Define the geometric functional
\begin{equation}
\label{Lambdamanip}
\Lambda[\overline{g}]\,:=\,\inf_{\Gamma}\,\left\{\mathcal{F}(\overline{g}, \overline{h}, \overline{\psi}) \right\}\,=\,\inf_{\Gamma}\,\int_M\,\left[\left(\mathrm{R}(\overline{g})\,+\,|\overline{\nabla}\, \overline{\psi}|^2\right)\,\overline{h}^2\,+\,
4\,|\overline{\nabla}\, \overline{h}|^2 \right]\,d\mu_{\overline{g}}\;.
\end{equation}
We have 
\begin{lemma}
\label{lemmaLambda}
On the Riemannian manifold $(M,g)$ there exists a pair of real numbers $\left(\lambda_1[\overline{g}],\,\lambda_2[\overline{g}]\right)$ and  a corresponding unique pair  $\left(\overline{h}_0,\,\overline{\psi}_0\right)\,\in\,\Gamma$ solution of the coupled elliptic eigenvalue problem 
\begin{eqnarray}
-\,4\,\Delta_{\overline{g}}\,\overline{h}_0\,+\,\left(\mathrm{R}(\overline{g})\,+\, 
|\overline{\nabla}\, \overline{\psi}_0|^2\,-\, \lambda_2[\overline{g}]\, \overline{\psi}_0^2  \right)\,\overline{h}_0&=&\lambda_1[\overline{g}]\,\overline{h}_0\;,\nonumber\\
\label{PDEsystem}\\
-\,\Delta_{\overline{g}}\,\overline{\psi}_0\,-\,\overline{\nabla}_i\,\ln\overline{h}^2_0\,\overline{\nabla}^i\,\overline{\psi}_0&=&
\lambda_2[\overline{g}]\,\overline{\psi}_0\nonumber\;,
\end{eqnarray}
such that for every pair $\left(\overline{h},\,\overline{\psi}\right)\,\in\,\Gamma$ we have
\begin{equation}
\label{infFeigen}
\int_M\,\left[\left(\mathrm{R}(\overline{g})\,+\,|\overline{\nabla}\, \overline{\psi}|^2\right)\,\overline{h}^2\,+\,
4\,|\overline{\nabla}\, \overline{h}|^2 \right]\,d\mu_{\overline{g}}\,\geq\,(\alpha_g)^{n/2}\left(\lambda_1[\overline{g}]\,+\,\lambda_2[\overline{g}]\right)\;,
\end{equation}
with equality for  $\left(\overline{h},\,\overline{\psi}\right)\,=\,\left(\overline{h}_0,\,\overline{\psi}_0\right)$.  
\end{lemma}
\begin{proof} (of Lemma \ref{lemmaLambda}).
In order to impose the normalization constraints $\int_M\,\overline{h}^2\,d\mu_{\overline{g}}\,=\,(\alpha_g)^{n/2}$ and  $\int_M\,\overline{\psi}^2\;\overline{h}^2\,d\mu_{\overline{g}}\,=\,(\alpha_g)^{n/2}$ in the variational characterization  (\ref{Lambdamanip}) of  ${\mathcal{F}}(\overline{g}, \overline{h}, \overline{\psi})$ we introduce two Lagrange multipliers $\lambda_1\in\mathbb{R}$ and $\lambda_2\in\mathbb{R}$ and consider the auxiliary functional
\begin{eqnarray}
&&\widehat{\mathcal{F}}(\overline{g}, \overline{h}, \overline{\psi};\,\lambda_1,\,\lambda_2)\nonumber\\
\label{Facca*}\\
&&\:=\,
\int_M\,\left[\left(\mathrm{R}(\overline{g})\,+\,|\overline{\nabla}\, \overline{\psi}|^2\right)\,\overline{h}^2\,+\,
4\,|\overline{\nabla}\, \overline{h}|^2\,-\,\lambda_1\left(\overline{h}^2-(\alpha_g)^{n/2}\right)\,-\,\lambda_2\left(\overline{\psi}^2\,\overline{h}^2
-(\alpha_g)^{n/2}\right) \right]\,d\mu_{\overline{g}}\;,\nonumber
\end{eqnarray} 
in terms of which we enforce the constraints by requiring $\frac{d}{d\lambda_1}\widehat{\mathcal{F}}=0=\frac{d}{d\lambda_2}\widehat{\mathcal{F}}$.
Under the stated assumptions,  $\widehat{\mathcal{F}}(\overline{g}, \overline{h}, \overline{\psi};\,\lambda_1,\,\lambda_2)$ is differentiable with respect to $(\overline{h}, \overline{\psi})$,  and by considering the germ of deformation of the pair $\left(\overline{h},\,\overline{\psi}\right)\,\in\,\mathrm{W}^{1,2}(M) \times\mathrm{W}^{1,2}(M)$ in the direction $\left(\rho ,\,\phi  \right)$ defined by
\begin{equation}
\left(\overline{h}_\epsilon,\,\overline{\psi}_\epsilon\right)\,=\,\left(\overline{h},\,\overline{\psi}\right)\,+\,\epsilon\,\left(\rho ,\,\phi  \right)\;,\;\;\;\;\rho,\;\phi\,\in\,C_0^\infty(M, \mathbb{R})\;,
\end{equation}
we can easily compute the corresponding Frechet derivative according to
\begin{eqnarray}
&&D\widehat{\mathcal{F}}\circ \left(\overline{h},\,\overline{\psi}\right)\,:=\,
\left. \frac{d}{d\epsilon}\,\widehat{\mathcal{F}}(\overline{g}, \overline{h}_\epsilon, \overline{\psi}_\epsilon;\,\lambda_1,\,\lambda_2)\right|_{\epsilon=0}\\
&&=\,-\,2\int_M\,\left[4\Delta _{\overline{g}}\,\overline{h}-  \left(\mathrm{R}(\overline{g})\,+\,|\overline{\nabla}\, \overline{\psi}|^2\,-\,\lambda_2\overline{\psi}^2\right)\,\overline{h}\,+\,\lambda_1\,\overline{h} \right]\rho\,d\mu_{\overline{g}}\nonumber\\
&&-\,2\int_M\,\left[\left(\Delta_{\overline{g}}\,\overline{\psi}\,+\,\lambda_2\overline{\psi}\right)\,\overline{h}^2
\,+\,\overline{\nabla}_i\,\overline{h}^2\,\overline{\nabla}^i\,\overline{\psi}\right]
\,\phi\,d\mu_{\overline{g}}\nonumber\\ 
&&=\,-\,2\int_M\,\left[4\Delta _{\overline{g}}\,\overline{h}-  \left(\mathrm{R}(\overline{g})\,+\,|\overline{\nabla}\, \overline{\psi}|^2\,-\,\lambda_2\overline{\psi}^2\right)\,\overline{h}\,+\,\lambda_1\,\overline{h} \right]\rho\,d\mu_{\overline{g}}\nonumber\\
&&\,-\,2\int_M\,\left(\Delta_{\overline{g}}\,\overline{\psi}\,+\,\overline{\nabla}_i\,\ln\overline{h}^2\,\overline{\nabla}^i\,\overline{\psi}+\,\lambda_2\overline{\psi}\right)\,\phi\,\overline{h}^2\,d\mu_{\overline{g}} \nonumber\;,
\end{eqnarray}
where, in the last line we have assumed $h>0$ in order to rewrite the term $\left[\ldots\,\overline{\nabla}_i\,\overline{h}^2\,\overline{\nabla}^i\,\overline{\psi}\right]\phi\,d\mu_{\overline{g}}$ in the more symmetric form $\left[\ldots\,\overline{\nabla}_i\,\ln\overline{h}^2\,\overline{\nabla}^i\,\overline{\psi}\right]\,\phi\,\overline{h}^2\,d\mu_{\overline{g}}$, featuring the integration with respect to the weighted measure $\overline{h}^2\,d\mu_{\overline{g}}$. Hence, if the identify the Lagrange multiplier $\lambda_1$ and $\lambda_2$ with $\lambda_1[\overline{g}]$ and $\lambda_2[\overline{g}]$  the variational characterization of  the functional (\ref{Lambdamanip})   is provided by the the elliptic eigenvalue problem(\ref{PDEsystem}), as stated. To establish the existence of a minimizer $(\overline{h}_0,\,\overline{\psi}_0)$ of (\ref{Lambdamanip}), with $\overline{h}_0\,>\,0$,\, we can proceed as follows. Let us take a minimizing sequence $\{(\overline{h}_K,\,\overline{\psi}_K)\}_{K=1}^\infty\,\in\,\Gamma$. Hence, the Sobolev norms $\parallel \overline{h}_K\parallel _{\mathrm{W}^{1,2}(M)}$ and $\parallel \overline{\psi}_K\parallel _{\mathrm{W}^{1,2}(M)}$ are bounded and we may choose a subsequence $\{(\overline{h}_i,\,\overline{\psi}_i)\}_{i=1}^\infty\,\in\,\Gamma$ of $\{(\overline{h}_K,\,\overline{\psi}_K)\}_{K=1}^\infty$ which converge to $(\overline{h}_0,\,\overline{\psi}_0)$, weakly in $\mathrm{W}^{1,2}(M)$ and strongly in $\mathrm{L}^{2}(M)$. We easily compute\footnote{See,\emph{e.g.}  \cite{flowers0} Lemma 5.22 for the analogous case of Perelman's $\lambda(g)$ functional}
\begin{eqnarray}
\int_M\,\left|\overline{\nabla}\left(\overline{h}_i-\overline{h}_0 \right)\right|^2\,d\mu_{\overline{g}}\,&=&\,
\int_M\,\left|\overline{\nabla}\,\overline{h}_i\right|^2\,d\mu_{\overline{g}}\,+\,
\int_M\,\left|\overline{\nabla}\,\overline{h}_0\right|^2\,d\mu_{\overline{g}}\,-\,
2\int_M\,\overline{\nabla}_a\,\overline{h}_i\,\overline{\nabla}^a\,\overline{h}_0\,d\mu_{\overline{g}}\,\geq\,0\nonumber\\
\\
&\Longrightarrow& \,\lim\inf_{i\rightarrow \infty}\,\int_M\,\left|\overline{\nabla}\,\overline{h}_i\right|^2\,d\mu_{\overline{g}}\,\geq\,\int_M\,\left|\overline{\nabla}\,\overline{h}_0\right|^2\,d\mu_{\overline{g}}\;,\nonumber
\end{eqnarray} 
where we have exploited the existence (by the weak convergence in $\mathrm{W}^{1,2}(M)$) of the limit
\begin{equation}
\lim_{i\rightarrow \infty}\,\int_M\,\overline{\nabla}_a\,\overline{h}_i\overline{\nabla}^a\,\overline{h}_0\,d\mu_{\overline{g}}\,=\,\int_M\,\left|\overline{\nabla}\,\overline{h}_0\right|^2\,d\mu_{\overline{g}}\;.
\end{equation}
By proceeding in a similar way and by exploiting  the strong convergence of $\{\overline{h}_i\}_{i=1}^\infty$ in $\mathrm{L}^2(M)$ we have
\begin{eqnarray}
\int_M\,\left|\overline{\nabla}\,\left(\overline{\psi}_i-\overline{\psi}_0 \right)\right|^2\,\overline{h}_i^2\,d\mu_{\overline{g}}\,&=&\,
\int_M\,\left|\overline{\nabla}\,\overline{\psi}_i\right|^2\,\overline{h}_i^2\,d\mu_{\overline{g}}\,+\,
\int_M\,\left|\overline{\nabla}\,\overline{\psi}_0\right|^2\,\overline{h}_i^2\,d\mu_{\overline{g}}\nonumber\\
&-&\,
2\int_M\,\overline{\nabla}_a\,\overline{\psi}_i\,\overline{\nabla}^a\,\overline{\psi}_0\,\overline{h}_i^2\,d\mu_{\overline{g}}\,\geq\,0\nonumber\\
\nonumber\\
&\Longrightarrow& \,\lim\inf_{i\rightarrow \infty}\,\int_M\,\left|\overline{\nabla}\,\overline{\psi}_i\right|^2\,\overline{h}_i^2\,d\mu_{\overline{g}}\,\geq\,\int_M\,\left|\overline{\nabla}\,\overline{\psi}_0\right|^2\,\overline{h}_0^2\,d\mu_{\overline{g}}\;,\nonumber
\end{eqnarray} 
Finally, the strong convergence of $\{\overline{h}_i\}_{i=1}^\infty$ in $\mathrm{L}^2(M)$ also implies the existence of the limits
\begin{eqnarray}
\lim_{i\rightarrow \infty}\,
\int_M\,\mathrm{R}(\overline{g})\,\overline{h}_i^2\,d\mu_{\overline{g}}\,=\,
\int_M\,\mathrm{R}(\overline{g})\,\overline{h}_0^2\,d\mu_{\overline{g}}\;,
\end{eqnarray} 
and 
\begin{eqnarray}
\lim_{i\rightarrow \infty}\,
\int_M\,\overline{h}_i^2\,d\mu_{\overline{g}}\,=\,
\int_M\,\overline{h}_0^2\,d\mu_{\overline{g}}\,=\,\left(\alpha_g \right)^{n/2}\;.
\end{eqnarray} 
It follows that $(\overline{h}_0,\,\overline{\psi}_0)$ is a minimizer of the functional (\ref{Lambdamanip}) among all $(\overline{h},\,\overline{\psi})\,\in\,\Gamma$, and provides a weak solution of the eigenvalue problem(\ref{PDEsystem}). Elliptic regularity allows to conclude that $(\overline{h}_0,\,\overline{\psi}_0)$ actually is a smooth solution. Unicity of the pair $(\overline{h}_0,\,\overline{\psi}_0)$ and positivity of $\overline{h}_0$ easily follows by standard arguments (which can be rather directly adapted from \cite{flowers0} Lemma 5.22 where they are discussed in detail for Perelman's $\lambda(g)$ functional). In particular, since
 $\overline{h}_0\,>\,0$ is smooth and $\int_M\,\overline{h}^2_0\,d\mu_{\overline{g}}\,=\,\left(\alpha_g \right)^{n/2}$, then there exists a unique smooth $\overline{f}_0\,:=\,-\,\ln\,\overline{h}_0^2$ such that $\int_M\,e^{-\,\overline{f}_0}\,d\mu_{\overline{g}}\,=\,\left(\alpha_g \right)^{n/2}$, and such that the pair $(\overline{f}_0,\,\overline{\psi}_0)$ is the unique smooth minimizer of the functional $\mathcal{F}(\overline{g}, \overline{f}, \overline{\psi})$. By multiplying both members of the first equation (\ref{PDEsystem}) by $\overline{h}_0$ and integrating the resulting expression  over $(M, g)$ we get
\begin{eqnarray}
\int_M\,\left[ -\,4\,\overline{h}_0\Delta_{\overline{g}}\,\overline{h}_0\,+\,\left(\mathrm{R}(\overline{g})\,+\, 
|\overline{\nabla}\, \overline{\psi}_0|^2\,-\, \lambda_2[\overline{g}]\,\overline{\psi}_0^2\,-\lambda_1[\overline{g}]  \right)\,\overline{h}^2_0\right]\,d\mu_{\overline{g}}\,=\,0\;.
\end{eqnarray} 
If we take into account the normalizations $\int_M\,\overline{h}^2_0\,d\mu_{\overline{g}}\,=\,(\alpha_g)^{n/2}\,=\,\int_M\,\overline{\psi}_0^2\,\overline{h}^2_0\,d\mu_{\overline{g}}$, we have
\begin{equation}
\int_M\,\left[\left(\mathrm{R}(\overline{g})\,+\,|\overline{\nabla}\, \overline{\psi}_0|^2\right)\,\overline{h}_0^2\,+\,
4\,|\overline{\nabla}\, \overline{h}_0|^2 \right]\,d\mu_{\overline{g}}\,=\,(\alpha_g)^{n/2}\left(\lambda_1[\overline{g}]\,+\,\lambda_2[\overline{g}]\right)\;,
\end{equation}
from which (\ref{infFeigen}) immediately follows.
\end{proof}

In order to discuss the geometrical meaning of the $\Lambda[\overline{g}]$ so characterized let us start by noticing that the eigenvalue equation $\Delta_{\overline{g}}\,\overline{\psi}_0\,+\,\overline{\nabla}_i\,\ln\overline{h}_0^2\,\overline{\nabla}^i\,\overline{\psi}_0\,+\,\lambda_2[\overline{g}]\,\overline{\psi}_0\,=\,0$ appearing in (\ref{PDEsystem}) is, in disguised form the eigenvalue problem \cite{colbois},  \cite{grigoryan} for the weighted Laplacian $\Delta ^{(\omega)}$ on the Riemannian manifold with density $(M,  \overline{g}, d\overline{\omega}\,:=\,e^{\,-\,\overline{f}_0}\,d\mu_{\overline{g}})$. \emph{i.e.} 
\begin{equation}
\label{eigval}
\Delta_{\overline{g}} ^{(\overline{\omega})}\,\overline{\psi}_{0}\,+\,\lambda_2[\overline{g}]\,\overline{\psi}_{0}\,=\,0\;.
\end{equation}
Hence,
\begin{equation}
\label{lambdaweighted}
\lambda_2[\overline{g}]\,:=\,\inf_{\overline{\psi}}\,\left\{\frac{\int_M\,\left|\overline{\nabla}\,\overline{\psi}  \right|^2\,e^{-\,\overline{f}_0}\,d\mu_{\overline{g}}}{\int_M\,\overline{\psi}^2\,e^{-\,\overline{f}_0}\,d\mu_{\overline{g}}}
\,:\,\overline{\psi}\,\in\,\mathrm{W}^{1,2}_{(\overline{\omega})}(M),\;\;\;\;
\int_M\,\overline{\psi}\,e^{-\,\overline{f}_0}\,d\mu_{\overline{g}}\,=\,0 \right\}\;, 
\end{equation}
where $\mathrm{W}^{1,2}_{(\overline{\omega})}(M)$ denotes the space of functions of Sobolev class $\mathrm{W}^{1,2}$ with respect to the measure $d\overline{\omega}\,:=\,e^{-\,\overline{f}_0}\,d\mu_{\overline{g}}$.
Also notice that from the  definition(\ref{orPerextended0})   of 
$\mathcal{F}(\overline{g}, {\overline{f}}, \overline{\xi}_g)$ we immediately obtain
\begin{equation}
\Lambda[\overline{g}]\,:=\,\inf_{\Gamma}\,\left\{\mathcal{F}(\overline{g}, {\overline{f}}) \,+\,\int_M\,\left|\overline{\nabla}\,\overline{\psi}\right|^2\,e^{\,-\,\overline{f}}\,d\mu_{\overline{g}} \right\}\;,
\end{equation}
where $\mathcal{F}(\overline{g}, {\overline{f}})$ is Perelman's energy  (\ref{orPer}). Hence,
\begin{eqnarray}
\Lambda[\overline{g}]\,&=&\,\inf_{\Gamma}\,\left\{\int_M\,\mathrm{R}^{Per}(\overline{g}(t))\,e^{\,-\,\overline{f}}\,d\mu_{\overline{g}}\,+\, \int_M\,\left|\overline{\nabla}\,\overline{\psi}\right|^2\,e^{\,-\,\overline{f}}\,d\mu_{\overline{g}}\right\}\\
&\geq&\,\inf_{\Gamma}\,\left\{\int_M\,\mathrm{R}^{Per}(\overline{g}(t))\,e^{\,-\,\overline{f}}\,d\mu_{\overline{g}}\right\}\,+\,\inf_{\Gamma}\left\{\int_M\,\left|\overline{\nabla}\,\overline{\psi}\right|^2\,e^{\,-\,\overline{f}}\,d\mu_{\overline{g}}\right\}\nonumber\\
&\geq&\,(\alpha_g)^{n/2}\,\lambda(\overline{g})\,+\,\inf_{\Gamma}\left\{\int_M\,\left|\overline{\nabla}\,\overline{\psi}\right|^2\,e^{\,-\,\overline{f}}\,d\mu_{\overline{g}}\right\}\;,\nonumber
\end{eqnarray}
where $\lambda(\overline{g})$ is Perelman's $\lambda$--functional,
\begin{equation}
\label{Perelambda}
\lambda(\overline{g})\,:=\,\inf_{\overline{f}}\,\left\{\int_M\,\mathrm{R}^{Per}(\overline{g}(t))\,e^{\,-\,\overline{f}}\,d\mu_{\overline{g}}\,:\,\overline{f}\in C^\infty(M,\mathbb{R}),\,\int_M\,e^{\,-\,\overline{f}}\,d\mu_{\overline{g}}\,=\,(\alpha_g)^{n/2} \right\}
\end{equation}
 (note that the normalization $\int_M\,e^{\,-\,\overline{f}}\,d\mu_{\overline{g}}\,=\,(\alpha_g)^{n/2}$ replaces the standard  $\int_M\,e^{\,-\,\overline{f}}\,d\mu_{\overline{g}}\,=\,1$ and alters the corresponding normalization of the eigenvalue as in (\ref{infFeigen})). Hence, $\Lambda[\overline{g}]$ dominates over Perelman's $\lambda(\overline{g})$-functional, which is not surprising since we are perturbing the scalar curvature $\mathrm{R}(\overline{g})$ with the positive term $\left|\overline{\nabla}\,\overline{\psi}\right|^2$.   Moreover, if we assume that the Bakry--Emery Ricci curvature (\ref{BEM}) of  $(M, \overline{g}, e^{\,-\,\overline{f}_0}\,d\mu_{\overline{g}})$ is bounded below according to 
\begin{equation}
\label{BEM***}
\mathrm{Ric}^{BE}(\overline{g})\,:=\,\mathrm{Ric}(\overline{g})\,+\,\overline{\nabla}\,\overline{\nabla}\,\overline{f}_0\,\geq\,\mathrm{C}_0\,\overline{g}\;,
\end{equation}
for some constant $\mathrm{C}_0\,\in\,\mathbb{R}$, (\emph{i.e.}, if $\left(M, \overline{g}, d\overline{\omega}:=e^{\,-\,\overline{f}_0}\,d\mu_{\overline{g}}\right)$ is a Bakry--Emery manifold), then by taking into account the lower bound estimate of the first eigenvalue of the weighted Laplacian over a compact Riemannian manifold with density (see Theorem 1.1  \cite{futaki2}, also \cite{futaki1},  \cite{andrews}), we get
\begin{equation}
(\alpha_g)^{n/2}\lambda_2(\overline{g})\,\geq\,\sup_{s\in (0,1)}\,\left\{4s(1-s)\,\frac{\pi^2}{\mathrm{diam}^2(\overline{g})}\,+\,s\,\mathrm{C}_0   \right\}\;,
\end{equation}
where $\mathrm{diam}^2(\overline{g})$ is the diameter of  $(M, \overline{g})$. Hence
\begin{equation}
\Lambda[\overline{g}]
\,=\,(\alpha_g)^{n/2}\left(\lambda_1[\overline{g}]\,+\,\lambda_2[\overline{g}]\right)\,\geq\,
(\alpha_g)^{n/2}\lambda(\overline{g})\,+\,\sup_{s\in (0,1)}\,\left\{4s(1-s)\,\frac{\pi^2}{\mathrm{diam}^2(\overline{g})}\,+\,s\,\mathrm{C}_0   \right\}\;,
\end{equation}
as stated.\\
\\ 
We conclude the proof of Theorem \ref{ThBB2} by showing how the monotonicity of 
 $\mathcal{F}(\overline{g}(t), {\overline{f}(t)}, \overline{\xi}_g(t))$ under the flows  $[0, T_0]\,\ni \,t\,\longmapsto\, (\bar g(t), d\overline{\omega}(t), \overline{\xi}_{g(t)})$ (\ref{gevol}),   (\ref{fevol}), and  (\ref{wabevol}) implies the monotonicity of the extended Perelman $\Lambda[\overline{g}]$--functional  under the RG-2 flow. To begin with, let $t_0\in [0, T_0]$ be given, and let $(\psi(t_0), f_0)$ be  the minimizers  of  $\mathcal{F}(\overline{g}(t_0), {\overline{f}(t_0)}, \overline{\xi}_g(t_0))$ as described above, so 
 \begin{equation}
\overline{\xi}_{g(t_0)}\ = \,\overline{\nabla}\,\overline{\psi}(t_0)\;,
\end{equation}
where $\overline{\nabla}$ is the gradient operator on $(M, \overline{g}(t_0))$. According to the minimization properties of  $(\psi(t_0), f_0)$
we have
 \begin{equation}
 \mathcal{F}(\overline{g}(t_0), {\overline{f}(t_0)}, \overline{\xi}_g(t_0)) \,=\,\Lambda[\overline{g}(t_0)]\;.
 \end{equation}
 Starting with the initial data $\left( \overline{g}(t_0), {\overline{f}(t_0)}, \overline{\xi}_g(t_0) \right)$ so obtained, we can solve in the time-reversed direction $\eta\,:=\,t_0-t$ the parabolic PDEs associated with the evolutions (\ref{fevol}) and (\ref{wabevol}), (the former, expressing the preservation of the measure $d\overline{\omega}(t)\,:=\,e^{-\,\overline{f}(t)}\,d\mu_{\overline{g}(t)}$, can be easily transformed into the corresponding parabolic PDE for  $\overline{f}(\eta)$).  According to Theorem \ref{ThBB},  along the  resulting flow $t\,\longmapsto \,\left( \overline{g}(t), {\overline{f}(t)}, \overline{\xi}_g(t) \right)$ we have 
 \begin{equation}
 \frac{d}{d t}\,\mathcal{F}(\overline{g}(t), {\overline{f}(t)}, \overline{\xi}_g(t))\,\geq\,0\;,
 \end{equation}
 for all $t\,\leq\,t_0$. This, together with  the preservation of the measure $d\overline{\omega}(t)$, directly implies that
 \begin{equation}
 \Lambda[\overline{g}(t)]\,\leq\,\mathcal{F}(\overline{g}(t), {\overline{f}(t)}, \overline{\xi}_g(t))\,\leq\,\mathcal{F}(\overline{g}(t_0), {\overline{f}(t_0)}, \overline{\xi}_g(t_0)) \,=\,\Lambda[\overline{g}(t_0)]\;.
 \end{equation}
 The monotonicity of  the functional $\Lambda[\overline{g}(t)]$ along the RG-2 flow can be easily obtained from this result if, following a standard procedure (see \textit{e.g.}  \cite{flowers0} Lemma 5.25), we consider the time derivative of  $\Lambda[\overline{g}(t)]$ in the sense of the $\liminf$ of backward difference quotients according to
 \begin{eqnarray}
 &&\left.\frac{d}{d t_-}\,\Lambda[\overline{g}(t)]\right|_{t_0}\,:=\,\liminf_{\epsilon\,\rightarrow \,0^+}\,
\frac{\Lambda[\overline{g}(t_0)]\,-\,\Lambda[\overline{g}(t_0\,-\,\epsilon)]}{\epsilon}\nonumber\\
\\
&&\geq\,\liminf_{\epsilon\,\rightarrow \,0^+}\,
\frac{\mathcal{F}(\overline{g}(t_0), {\overline{f}(t_0)}, \overline{\xi}_g(t_0))\,-\,
\mathcal{F}(\overline{g}(t_0-\epsilon), {\overline{f}(t_0-\epsilon)}, \overline{\xi}_g(t_0-\epsilon))}{\epsilon}
\,=\,
\left.\frac{d}{d t}\,\mathcal{F}(\overline{g}(t), {\overline{f}(t)}, \overline{\xi}_g(t))\right|_{t_0}\;.\nonumber
 \end{eqnarray}
 Hence, from the relation (\ref{nearMiss0})  we get
 \begin{equation}
 \left.\frac{d}{d t_-}\,\Lambda[\overline{g}(t)]\right|_{t_0}\,\geq\,\,2\,\int_M\,\left| \mathrm{Ric}^{BE}(\overline{g}(t_0))\,+\,\frac{\alpha_g}{8}\mathrm{Rm}^2(\overline{g}(t_0), \overline{\xi}_{g(t_0)})\right|^2_{\overline{g}(t_0)}\,d\overline{\omega} (t_0)\;.
 \end{equation}
 Since    the choice of $t_0\,\in\,[0,T_0]$ is arbitrary, we eventually get
\begin{equation}
 \frac{d}{d t_-}\,\Lambda[\overline{g}(t)]\,\geq\,\,2\,\int_M\,\left| \mathrm{Ric}^{BE}(\overline{g}(t))\,+\,\frac{\alpha_g}{8}\mathrm{Rm}^2(\overline{g}(t), \overline{\xi}_{g(t)})\right|^2_{\overline{g}(t)}\,d\overline{\omega} (t)\;,
 \end{equation}
 for all $t\,>\,0$ for which the RG-2 flow exists,  as stated. 
\end{proof}

\begin{rem}
The structure of the above proof strongly suggests that a similar monotonicity result should work also for the geometric flow  associated with the higher loop  approximations  to the perturbative renormalization group flow for non--linear sigma model. This is a largely uncharted territory, and already   proving a local existence result for the geometric flows associated to the $3$--loop and $4$--loop curvature contributions, (where explicit curvature expressions are available--see \emph{e.g.} \cite{tseytlin})  is an extremely demanding task. Were this  possible, one could  presumably use the parabolic Fokker--Planck evolution  (\ref{wabevol}) (with the quadratic source term $\frac{\alpha_g^2}{32}\,\left|\mathrm{Rm}^2(\overline{g}(t),\,\overline{\xi}_{g(t)})\right|^2$ replaced by the corresponding $k$--th order curvature terms present at the given loop approximation)  in order to control the monotonicity of the associated  energy functional. 
\end{rem}

\section{IS THE RG-2 FLOW A GRADIENT FLOW?}

The  expression  (\ref{PerRG2F})  directly shows that  although the modified Perelman entropy 
\begin{equation}
\label{firstModPer}
\mathcal{F}_{(1)}(\overline{g}, d\overline{\omega},\,\overline{\xi}_{g})\,:=\,
\int_M\, \left(\mathrm{R}^{Per}(\overline{g}(t))\,+\, \left|\overline{\xi}_{g(t)}\right|^2\right)\,d\overline{\omega} (t)\;,
\end{equation} 
is monotonic, the (modified) RG--2 flow (\ref{DToutPer}) is not a gradient flow with respect this functional. Actually,
 the functional with respect to which the  (DeTurck modified)  RG--2 flow  is gradient   is a rather non--trivial  modification of  (\ref{firstModPer}):
\begin{equation}
\label{extPerelman}
\mathcal{F}_{(2)}(\overline{g}, d\overline{\omega},\,\overline{\xi}_{g})\,:=
\,\int_M\,\left[\mathrm{R}^{Per}(\overline{g}(t))\,+\,\frac{\alpha_g}{8}\,\left|\mathrm{Rm}(\overline{g}(t))\right|^2_{g(t)}\,-\,
\mathrm{div}_{\overline{g}(t)}\overline{\xi}_{g(t)}\right]\,d{\omega}(t)\;.
\end{equation}
In order to simplify the computation of  $\frac{d}{d t}\mathcal{F}_{(2)}(\overline{g}, \overline{f}, \overline{\xi})$
and avoid the annoying overlines $\overline{..A..}$ induced by pulling back the RG-2 flow metric back and forth, we abuse notation and drop the overlines, with the proviso that everything refers to the DeTurck modified RG-2 flow 
 \begin{equation} 
	 \label{DTout1no-}
\frac{\partial }{\partial t }\,\overline{g}(t)\,=\,-\,
2\,\mathrm{Ric}^{BE}(\overline{g}(t ))\,-\,\frac{\alpha_g}{2}\,\mathrm{Rm}^2(\overline{g}(t), \overline{\xi}_{\overline{g}(t)}) \;. 
\end{equation}
 (See (\ref{DToutPer}). Obviously the pull--back in (\ref{DTout1no-}) is not generated by the same $f$ and $\xi_g$ featuring in (\ref{DToutPer})).  
 
 To begin, 
we  remark that along a generic (germ of) curve of metrics $[0,1]\,\ni \,t\,\longmapsto \,g(t)$ with tangent vector $v\in C^\infty(M, \otimes ^2_{sym}T^*M)$
\begin{equation}
\label{tangentv}
\frac{\partial}{\partial t}\,g_{jk}(t)\,=\,v_{jk}\;,
\end{equation}
we have
\begin{equation}
\label{rm2der}
\frac{\partial}{\partial t}\,|\mathrm{Rm}(t)|^2\,=\,-\,4\,\mathrm{R}_{ijkl}\nabla^i \nabla^l v^{jk}\,-\,
2\,\mathrm{R}m^2_{jk}\,v^{jk}\;.
\end{equation} 
We have, from  (\ref{rm2der}) and  (\ref{Perlequation0}),  the pointwise evolution 
\begin{eqnarray}
\label{Perlequation1}
&&\frac{\partial }{\partial t}\,\left[\mathrm{R}^{Per}(g(t))\,+\,\frac{\alpha_g}{8}\,|\mathrm{R}m(t)|^2\right]=\,\nabla ^{(\omega)}_j\nabla ^{(\omega)}_k\,v^{jk}\,-\,\frac{\alpha_g}{2}\,\mathrm{R}_{ijkl}\nabla^i \nabla^l v^{jk}\nonumber\\
\\
&&-\,\left(\mathrm{R}^{BE}_{jk}\, +\,\frac{\alpha_g}{4}\, \mathrm{R}m^2_{jk}  \right)\,v^{jk}\,
+\,2\,\triangle _g^{(\omega)}\,\left(\frac{\partial f}{\partial t}-\frac{1}{2}\,tr_g(v) \right)\nonumber\;.
\end{eqnarray}
Let us recall that along (\ref{tangentv}) we have the relation (\ref{comput1}), hence, 
if  in line with the preservation of the pull-back measure $d\,\overline{\omega}(\beta)$   we assume  the measure preserving condition $\ppt\,d\omega(t)\,=\,0$,  and take into account the integration by parts formula 
\begin{equation}
\int_M\,\mathrm{R}_{ijkl}\nabla^i \nabla^l v^{jk}(t)\,d\omega\,=\,
\int_M\,\nabla^l_{(\omega)}\nabla^i_{(\omega)}\mathrm{R}_{ijkl}\, v^{jk}(t)\,d\omega\;,
\end{equation} 
then we easily get  from (\ref{Perlequation1})

\begin{eqnarray}
&&\frac{d}{dt}\mathcal{F}_{(2)}({g}, {f}\, {\xi})\,=\,-\,\int_M\,\left(\mathrm{R}_{jk}^{BE}({g}(t))\,+\,
\frac{\alpha_g}{4}\,\mathrm{Rm}_{jk}^2({g}(t), \xi_{g(t)})
 \right)\,v^{jk}\,d{\omega}(t)\nonumber\\
 \label{longcomp}\\
 &&-\,\frac{1}{2}\int_M\,g^{ab}(t)\left[\frac{\partial}{\partial t}\mathrm{L}_\xi g_{ab}\,+\,{\alpha_g}\,
{\nabla}_a^{\,(\omega)}{\nabla}_{\,(\omega)}^i\,\mathrm{R}_{ijkb}({g}(t))\,v^{jk}
\right]\,d{\omega}(t)\nonumber\;.
\end{eqnarray}
Since  $\int_M\,g^{ab}\,\Delta_{g(t)}^{(\omega)}\,\mathrm{L}_{\xi(t)} g_{ab}(t)\,d\omega(t)\,=\,0$, we can conveniently rewrite
this expression as 
\begin{eqnarray}
&&\frac{d}{dt}\mathcal{F}_{(2)}({g}, {f},\, {\xi}_g)\,=\,-\,\int_M\,\left(\mathrm{R}_{jk}^{BE}({g}(t))\,+\,
\frac{\alpha_g}{4}\,\mathrm{Rm}_{jk}^2({g}(t), \xi_{g(t)}) 
 \right)\,v^{jk}\,d{\omega}(t)\nonumber\\
 \label{longcomp}\\
 &&-\,\frac{1}{2}\int_M\,g^{ab}(t)\left[\frac{\partial}{\partial t}\mathrm{L}_\xi g_{ab}\,
-\,\Delta_{g(t)}^{(\omega)}\,\mathrm{L}_{\xi(t)} g_{ab}(t)\,
+\,{\alpha_g}\,
{\nabla}_a^{\,(\omega)}{\nabla}_{\,(\omega)}^i\,\mathrm{R}_{ijkb}({g}(t))\,v^{jk}
\right]\,d{\omega}(t)\nonumber\;,
\end{eqnarray}
which directly implies the  following result, where we have set $RG\,:=\,-\,2\mathrm{Ric}^{BE}({g})\,-\,\frac{\alpha_g}{2}\,\mathrm{Rm}^2({g}, \xi_g)$. 
\begin{theorem}[Entropy]  
\label{Entropy}
The coupled DeTurck RG-2 flow $[0, T_0]\,\ni\,t\,\longmapsto \, \left(g(t),\,d\omega(t),\,\xi_{g(t)}\right)$ solution of 
\begin{eqnarray} 
	 \label{DTcoupled}
\frac{\partial }{\partial t }\,{g}_{ij}(t)\,&=&\,-\,
2\mathrm{Ric}^{BE}_{ij}({g}(t ))\,-\,\frac{\alpha_g}{2}\,\mathrm{Rm}^2_{ij}({g}(t), \xi_{g(t)}) \;, \nonumber\\
\frac{\partial}{\partial t}\,d\omega(t)\,&=&\,0\;,\\
\frac{\partial}{\partial t}\mathrm{L}_\xi g_{ab}\,&=&
\,\Delta_{g(t)}^{(\omega)}\,\mathrm{L}_{\xi(t)} g_{ab}(t)\,
-\,{\alpha_g}\,
{\nabla}_a^{\,(\omega)}{\nabla}_{\,(\omega)}^i\,\mathrm{R}_{ijkb}({g}(t))\,RG^{jk}\nonumber\;,
\end{eqnarray}
is the gradient flow of the functional $\mathcal{F}_{(2)}({g}, {f}\, {\xi})$.
\end{theorem}
Since the term ${\nabla}_a^{\,(\omega)}{\nabla}_{\,(\omega)}^i\,\mathrm{R}_{ijkb}$ in  (\ref{DTcoupled}) gives rise to such a strong non--linear coupling among the $g(t)$ and $\xi_g(t)$ evolution, it would seem difficult to explicitly characterize a diffeomorphism   that pulls back the solution of  (\ref{DTcoupled}) to a standard RG-2 flow.

To better understand the geometric nature of  this latter  remark,  let us introduce Hamilton's Harnack quadric ( see \cite{flowers}, p. 32) 
\begin{eqnarray}
\label{RHarnack}
H_{\nabla f}\,&=&\, e^{f}\,\left(div\circ div\,+\,\mathrm{R}ic\,+\,\nabla \nabla \,f\right)_{1,4}\,\left(e^{\,-f}\,\mathrm{R}m\right)\\
&=&\,e^{f}\,\left(\nabla ^l\,\nabla ^i\,+\,\mathrm{R}^{li}\,+\,\nabla^l\nabla^i\,f\right)
\left(e^{\,-f}\,\mathrm{R}_{ijkl}\right)\;,\nonumber
\end{eqnarray}
where the subscript $(\ldots)_{1,4}$ denote the components of the Riemann tensor on which the operator between brackets is acting. Since
\begin{equation}
e^{f}\,\nabla ^l\,\nabla ^i\,\left(e^{\,-f}\,\mathrm{R}_{ijkl}\right)\,=\,
e^{f}\,\nabla ^l\,\left[e^{\,-f}\,e^{\,f}\,\nabla ^i\,\left(e^{\,-f}\,\mathrm{R}_{ijkl}\right)\right]
=\nabla _{(\omega)}^l\,\nabla _{(\omega)}^i\,\mathrm{R}_{ijkl}\;,
\end{equation}
we can equivalently write (\ref{RHarnack}) as 
\begin{equation}
\label{Harn}
\nabla_{(\omega)} ^l\,\nabla_{(\omega)} ^i\,\mathrm{R}_{ijkl}\,=\,\left(H_{\nabla f}\right)_{jk}\,
-\,(\mathrm{R}^{BE})^{li}\,\mathrm{R}_{ijkl}\;.
\end{equation}
Note that a long but straightforward computation   provides 
\begin{equation}
\label{nicewrite2}
\nabla_{(\omega)}^l \nabla_{(\omega)}^i \mathrm{R}_{ijkl}\,=\,\triangle ^{(\omega)} \mathrm{R}^{BE}_{kj}\,-\,g^{ab}\mathrm{R}^{BE}_{ka}\mathrm{R}^{BE}_{bj}\,+\,\mathrm{R}_{ijkl}(\mathrm{R}^{BE})^{il}\,-\,\frac{1}{2}\,L_{X}g_{kj}\;,
\end{equation}
where $X_h\,:=\,\frac{1}{2}\,\nabla_h\,\mathrm{R}^{Per}$. In particular  we can rewrite the Harnack quadric (\ref{Harn}) as
\begin{equation}
\label{Harnii}
\left(H_{\nabla f}\right)_{jk}\,=\,
\triangle ^{(\omega)} \mathrm{R}^{BE}_{kj}\,-\,g^{ab}\mathrm{R}^{BE}_{ka}\mathrm{R}^{BE}_{bj}\,+\,2\mathrm{R}_{ijkl}(\mathrm{R}^{BE})^{il}\,-\,\frac{1}{2}\,\nabla_j\nabla_k\,\mathrm{R}^{Per}\;,
\end{equation}
which, for $f=0$, reduces to the standard expression
\begin{equation}
\label{MHarnii}
M_{jk}\,:=\,
\triangle \mathrm{R}_{kj}\,-\,\mathrm{R}_{k}^l\mathrm{R}_{lj}\,+\,2\mathrm{R}_{ijkl}\mathrm{R}^{il}\,-\,\frac{1}{2}\,\nabla_j\nabla_k\,\mathrm{R}\;,
\end{equation}
featuring (up to the term $(2t)^{-1}\,\mathrm{R}_{kj}$) in the analysis of Hamilton's Harnack inequality. Note also that
if $(M,g,\,d\omega)$ is a Ricci soliton, \emph{i.e.} if
\\
\begin{equation}
\mathrm{R}^{BE}_{kl}\,:=\mathrm{R}_{kl}\,+\,\nabla_k \nabla_l\,f\, =\,\frac{\epsilon }{2}\,g_{kl}\;, \;\;\;\;\;\epsilon \in\,\mathbb{R}\;,
\end{equation}
we have
\begin{equation}
\nabla_{(\omega)}^l \nabla_{(\omega)}^i \mathrm{R}_{ijkl}\,=\,0\;.
\end{equation}
Indeed,  for a Ricci soliton one easily computes
\begin{eqnarray}
&&\triangle ^{(\omega)} \mathrm{R}^{BE}_{kj}\,-\,g^{ab}\mathrm{R}^{BE}_{ka}\mathrm{R}^{BE}_{bj}\,+\,\mathrm{R}_{ijkl}(\mathrm{R}^{BE})^{il}\nonumber\\
&&=\,-\,\frac{\epsilon^2}{4}\,g_{kj}\,+\,\frac{\epsilon}{2}\,\mathrm{R}_{kj}\,=\,-\,\frac{\epsilon^2}{4}\,g_{kj}\,+\,\frac{\epsilon}{2}\,\left(\frac{\epsilon}{2}\,g_{kj}\,-\,\nabla_k\nabla_j\,f\right)\nonumber\\
&&=\,-\,\frac{\epsilon}{2}\,\nabla_k\nabla_j\,f\;.
\end{eqnarray}
On the other hand, for $\mathrm{R}^{BE}_{kl}\,=\,\frac{\epsilon }{2}\,g_{kl}$ one has
\begin{equation}
X_j\,:=\,\nabla^l_{(\omega)}\,{R}^{BE}_{lj}\,=\,-\,\frac{\epsilon}{2}\,\nabla_j\,f\;,
\end{equation}
so that 
\begin{equation}
\frac{1}{2}\,L_X\,g_{kj}\,=\,-\,\frac{\epsilon}{2}\,\nabla_k\nabla_j\,f\;,
\end{equation}
and 
\begin{eqnarray}
&&\nabla-{(\omega)}^l \nabla_{(\omega)}^i \mathrm{R}_{ijkl}\,=\,
\left[\triangle ^{(\omega)} \mathrm{R}^{BE}_{kj}-\,g^{ab}\mathrm{R}^{BE}_{ka}\mathrm{R}^{BE}_{bj}\right.\\
&&\left.+\,\mathrm{R}_{ijkl}(\mathrm{R}^{BE})^{il}\,
-\,\frac{1}{2}\,L_X\,g_{kj}\,\right]_{\mathrm{R}ic^{BE}\,=\,\frac{\epsilon }{2}\,g}=0\;.\nonumber
\end{eqnarray} 

From these remarks it directly follows that it is the \textit{extended} Harnack  term $\nabla_{(\omega)}^l \nabla_{(\omega)}^i \mathrm{R}_{ijkl}$ that makes the gradient flow nature of the RG--2 flow so complex. This is quite manifest if we set $\xi_g\,=\,0$ in (\ref{longcomp})  to get  
\begin{equation}
\frac{d}{dt}\mathcal{F}_{(2)}({g}, {f})\,=\,-\,\int_M\,\left(\mathrm{R}_{jk}^{BE}({g}(t))\,+\,
\frac{\alpha_g}{4}\,\mathrm{Rm}_{jk}^2({g}(t)) \,-\,\frac{\alpha_g}{2}\,
{\nabla}^l_{\,(\omega)}{\nabla}_{\,(\omega)}^i\,\mathrm{R}_{ijkl}({g}(t))
 \right)\,v^{jk}\,d{\omega}(t)\;,
 \label{longcompbis}
\end{equation}
which is monotonic  along Ricci solitons, and also shows in a rather direct way that it  is the (fourth-order) flow
\begin{equation}
\frac{\partial}{\partial t}\,g_{jk}(t)\,=\,-\,2 \mathrm{R}_{jk}^{BE}({g}(t))\,-\,
\frac{\alpha_g}{2}\,\mathrm{Rm}_{jk}^2({g}(t)) \,+\,{\alpha_g}\,
{\nabla}^l_{\,(\omega)}{\nabla}_{\,(\omega)}^i\,\mathrm{R}_{ijkl}({g}(t))
\end{equation}
that is formally  the gradient flow of the functional  $\mathcal{F}_{(2)}({g}, {f},\,\xi_g)$ for $\xi_g\,=\,0$. It is only by taming the Harnack term ${\nabla}^l_{\,(\omega)}{\nabla}_{\,(\omega)}^i\,\mathrm{R}_{ijkl}({g}(t))$ by introducing the  evolution of   $\xi_g$ provided by the inhomogeneous heat equation
\begin{equation}
\frac{\partial}{\partial t}\mathrm{L}_\xi g_{ab}\,=
\,\Delta_{g(t)}^{(\omega)}\,\mathrm{L}_{\xi(t)} g_{ab}(t)\,
-\,{\alpha_g}\,
{\nabla}_a^{\,(\omega)}{\nabla}_{\,(\omega)}^i\,\mathrm{R}_{ijkb}({g}(t))\,RG_{jk}\;,
\end{equation}
(see (\ref{DTcoupled}))
that one can make manifest the gradient-like nature of the RG-2 flow with respect to $\mathcal{F}_{(2)}({g}, {f},\,\xi_g)$.

\end{document}